\documentclass[noinfoline]{imsart}

\RequirePackage[OT1]{fontenc}
\RequirePackage{amsthm,amsmath}
\RequirePackage[colorlinks,citecolor=blue,urlcolor=blue]{hyperref}
\usepackage{amsmath}
\usepackage{amsfonts}
\usepackage{amssymb}
\usepackage{hyperref}
\usepackage{amsthm}
\usepackage{color}
\usepackage{bbm}
\usepackage{graphicx}
\usepackage{caption}
\usepackage{subcaption}
\usepackage[longnamesfirst,round]{natbib}
\usepackage{color}
\usepackage[normalem]{ulem}
\usepackage{algorithmic}
\usepackage{algorithm}

\DeclareMathOperator{\trace}{trace}
\DeclareMathOperator{\argmin}{arg \ min}
\DeclareMathOperator{\rank}{rank}
\DeclareMathOperator{\diag}{diag}
\DeclareMathOperator{\Var}{Var}

\newtheorem{lemma}{Lemma}[section]
\newtheorem{corollary}{Corollary}[section]
\newtheorem{theorem}{Theorem}[section]

\newtheorem{remark}{Remark}

\newtheorem{assumption}{Assumption}

\startlocaldefs
\numberwithin{equation}{section}
\theoremstyle{plain}

\endlocaldefs

\begin{document}

\begin{frontmatter}
\title{Robust Low-Rank Matrix Estimation}
\runtitle{Robust Low-Rank Matrix Estimation}

\begin{aug}
\author{\fnms{Andreas} \snm{Elsener}\ead[label=e1]{elsener@stat.math.ethz.ch}}
\and
\author{\fnms{Sara} \snm{van de Geer}\ead[label=e2]{geer@stat.math.ethz.ch}}

\runauthor{Elsener and van de Geer}

\affiliation{ETH Z\"urich}

\address{
Seminar for Statistics\\
ETH Zurich\\
8092 Zurich \\
Switzerland \\
\printead{e1}\\
\printead{e2}}

\end{aug}

\begin{abstract}
Many results have been proved for various nuclear norm penalized estimators of the uniform sampling matrix completion problem. However, most of these estimators are not robust: in most of the cases the quadratic loss function and its modifications are used. We consider robust nuclear norm penalized estimators using two well-known robust loss functions: the absolute value loss and the Huber loss. Under several conditions on the sparsity of the problem (i.e. the rank of the parameter matrix) and on the regularity of the risk function sharp and non-sharp oracle inequalities for these estimators are shown to hold with high probability. As a consequence, the asymptotic behavior of the estimators is derived. Similar error bounds are obtained under the assumption of weak sparsity, i.e. the case where the matrix is assumed to be only approximately low-rank. In all our results we consider a high-dimensional setting. In this case, this means that we assume $n\leq pq$. Finally, various simulations confirm our theoretical results. 
\end{abstract}



\end{frontmatter}

\section{Introduction}
\label{s:introduction}
\subsection{Background and Motivation}
Netflix, Spotify, Apple Music, Amazon and many other on-line services offer an almost infinite amount of songs or films to their users. Clearly, a single person will never be able to watch every film or to listen to every available song. For this reason, an elaborate recommendation system is necessary in order to allow the users to choose content that already match his or her preferences. Many models and estimation methods have been proposed to address this question. Matrices provide an appropriate way of modelling this problem. Imagine that the plethora of films/songs is identified with the rows of a matrix, call it $B^*$, and the users with its columns. One entry of the matrix corresponds to the rating given to film $`` i" $ (row) by user $``j"$ (column). This matrix will have many missing entries. These entries are bounded and we can expect the rows of $B^*$ to be very similar to each other. It is therefore sensible to assume that $B^*$ has a low rank. The challenge is now to predict the missing ratings/fill in the empty entries of $B^*$. Define for this purpose the set of observed (possibly noisy) entries
\begin{align}
\mathcal{A} := \left\lbrace (i,j) \in \left\lbrace 1, \dots, p \right\rbrace \times \left\lbrace 1, \dots, q \right\rbrace : \mbox{the (noisy) entry} \right. \\ \left. \ A_{ij}\ \mbox{of} \ B^* \ \mbox{is observed} \right\rbrace,\nonumber
\end{align}
where $p$ is the number of films/songs and $q$ the number of users. Our estimation problem can therefore be stated in the following way: for $B \in \mathcal{B} \subset \mathbb{R}^{p \times q}$ we 
\begin{equation*}
\mbox{minimize} \ R_n(B), \ \mbox{subject to} \ \mbox{rank}(B) = s.
\end{equation*}
In this special case we have that 
\begin{equation}
\label{eqn:setb}
\mathcal{B} = \left\lbrace B \in \mathbb{R}^{p \times q} \vert \Vert B \Vert_\infty \leq \eta \right\rbrace,
\end{equation}
where $\eta$ is for instance the mean highest rating, and $R_n$ is some convex empirical error measure that is defined by the data, e.g.
\begin{equation}
R_n (B) = \frac{1}{\vert \mathcal{A} \vert}\sum_{(i,j) \in \mathcal{A}} (A_{ij} - B_{i,j})^2.
\end{equation}
Since the rank of a matrix is not convex we use the nuclear norm as its convex surrogate. This leads us to a relaxed convex optimization problem. For $B \in \mathcal{B}$ we
\begin{equation*}
\mbox{minimize} \ R_n(B), \ \mbox{subject to} \ \Vert B \Vert_{\mbox{nuclear}} \leq \tau,
\end{equation*}
for some $\tau > 0$.
The model described above can be considered as a special case of the trace regression model.

In the \textit{trace regression model} (see e.g. \cite{rohde2011estimation}) one considers the observations $(X_i, Y_i)$ satisfying
\begin{equation}
\label{eqn:traceregr}
Y_i = \trace (X_i B^*) + \varepsilon_i, \ i = 1, \dots, n,
\end{equation}
where $\varepsilon_i$ are i.i.d. random errors.
The matrices $X_i$ are so-called masks. They are assumed to lie in
\begin{equation}
\label{eqn:space}
\mathcal{\chi} = \left\lbrace e_k(q) e_l(p)^T : 1 \leq k \leq q, 1 \leq l \leq p  \right\rbrace,
\end{equation}
where $e_k(q)$ is the $q$-dimensional $k$-th unit vector and $e_l(p)$ is the 
$p$-dimensional $l$-th unit vector. We will assume that the $X_i$ are i.i.d. with
\begin{equation*}
\mathbb{P} \left(X_{i_{kj}} = 1 \right) = 1- \mathbb{P} \left(X_{i_{kj}} = 0 \right) = \frac{1}{pq}
\end{equation*} for all $i \in \left\lbrace 1, \dots, n \right\rbrace$, $k \in \left\lbrace 1, \dots q \right\rbrace$, and $j \in \left\lbrace 1, \dots, p \right\rbrace$. However, we point out that it is not necessary for our estimators to know this distribution. This knowledge will only be used in the proofs of the theoretical results.

The trace regression model together with the space $\chi$ and the distribution on $\chi$ is equivalent to the matrix completion case. The entries of the vector $Y$ can be identified with the observed entries as those in the matrix $A$.

From this, it can be seen that we are in a high-dimensional setting since the number of observations $n$ must be smaller than or equal to the total number of entries of $A$.
The setup described above is then called \textit{uniform sampling matrix completion}. A very similar setup was first considered in  \cite{srebro2004maximum} and \cite{srebro2005rank}.

As in the standard regression setting, parameter estimation in the trace regression model is also done via empirical risk minimization. Using the Lagrangian form for $B \in \mathcal{B}$ we 
\begin{equation*}
\mbox{minimize} \ R_n(B) + \lambda \Vert B \Vert_{\mbox{nuclear}},
\end{equation*}
where $R_n (B) = 1/n \sum_{i=1}^n \rho(Y_i - \trace( X_i B))$, $\rho$ is a convex loss function and $\lambda > 0$ is the tuning parameter. The loss function is often chosen to be the quadratic loss (or one of its modifications) as in \cite{Koltchinskii2011, negahban2011estimation, negahban2012restricted, rohde2011estimation} and many others. In \cite{lafond2015low} the case of an error distribution belonging to an exponential family is considered. As long as the errors are assumed to be light tailed as it is the case for i.i.d. Gaussian errors the least squares estimator will perform very well. However, the ratings are heavily subject to frauds (e.g. by the producer of a film). It is necessary to take this fact into account also in the estimation procedure. One might also be interested in estimating the median or another quantile of the ratings. For this purpose, M-estimators based on different losses than the quadratic loss are usually chosen.

\subsection{Proposed estimators}

In this paper, we consider the absolute value loss and the Huber loss.
The first robust estimator is then given by
\begin{equation}
\label{eqn:robust}
\hat{B} := \underset{B \in \mathcal{B}}{\argmin} \ \frac{1}{n} \sum_{i=1}^n \left\vert Y_i - \trace\left(X_i B\right) \right\vert + \lambda \left\Vert B \right\Vert_{\mbox{nuclear}}.
\end{equation}
Using the Huber loss we can define
\begin{equation}
\label{eqn:huber}
\hat{B}_H := \underset{B \in \mathcal{B} }{\argmin} \frac{1}{n} \sum_{i=1}^n \rho_H \left( Y_i - \trace \left( X_i B \right) \right) + \lambda \left\Vert B\right\Vert_{\mbox{nuclear}},
\end{equation}
where the function
\begin{equation*}
\rho_H(x) := \left\lbrace \begin{array}{ll}
x^2, & \mbox{if} \ \left\vert x \right\vert \leq \kappa \\
2 \kappa \left\vert x \right\vert - \kappa^2, & \mbox{if} \  \left\vert x \right\vert > \kappa
\end{array} \right.
\end{equation*}
defines the Huber loss function. The tuning parameter $\kappa > 0 $ is assumed to be given for our estimation problem. The possible values for the Huber parameter  $\kappa$  depend on the distribution of the errors as shown in Lemma \ref{lemma:2margin}. In practice, one usually estimates $\kappa$ and $\lambda$ with methods such as cross-validation.
Notice that it could happen that the estimators defined in Equations \ref{eqn:robust} and \ref{eqn:huber} are not unique since the objective functions are not strictly convex.
As will be shown, the rates depend on the Lipschitz constants of the loss functions and on $\eta$. Typically, the Lipschitz constants of the absolute value loss as well as of the Huber loss induce smaller constants in the rates compared to the Lipschitz constant of the truncated quadratic loss.

The ``target'' is defined as 
\begin{equation*}
B^0 := \underset{B \in \mathcal{B}}{\argmin} \ R(B),
\end{equation*}
where $R(B) = \mathbb{E} R_n (B)$ is the theoretical risk. It has to be noticed that the matrix $B^0$ is not necessarily equal to the matrix $B^*$.
Our main interest lies in the theoretical analysis of the above estimators. The estimators should mimic the sparsity behavior of an oracle $B$. In the case of the absolute value loss we will prove a non-sharp oracle inequality. Whereas for the Huber loss, thanks to its differentiability, we are able to derive a sharp oracle inequality.

Assuming $B=B^0$ in Corollary \ref{cor:usrcompletionhuber} the upper bound is typically of the form
\begin{equation}
R(\hat{B}_H) - R(B^0) \lesssim \lambda^2 pqs_0 ,
\end{equation}
where $ \lesssim $ means that some multiplicative constants (depending on the tuning parameter $\kappa$) are omitted.

The assumptions for this kind of results are mainly based on the regularity of the absolute value and the Huber loss. Moreover, the properties of the nuclear norm, which are very similar to those of the $\ell_1$-norm for vectors, will be exploited.  In addition, we use the sparsity behavior induced by the nuclear norm to infer the behavior of weakly sparse estimators. This takes into account that a matrix could have few very large singular values and many small, but not exactly zero singular values:
\begin{equation*}
\ B \in \left\lbrace B' \in \mathcal{B} : \sum_{j =1}^q \Lambda_j^r \leq \rho_r^r, \Lambda_1, \dots, \Lambda_q \ \mbox{the singular values of}  \ B'	\right\rbrace,
\end{equation*}
where $0 < r <1$ and $\rho_r^r$ is some reasonably small constant.

\subsection{Related Literature}

A first study with robust matrix estimation was made in \cite{chandrasekaran2011rank} in a setting with no missing entries. In order to avoid identifiability issues the authors introduce ``incoherence'' conditions on the low-rank component. These conditions make sure that the low-rank component itself is not too sparse. The locations of the corruptions are assumed to be fixed. In the context of Principal Component Analysis (PCA) which is a special case of the matrix regression model robustness was investigated in \cite{candes2011robust}. The authors assume that the matrix to be estimated is decomposed in a low-rank matrix and a sparse matrix. In contrast to \cite{chandrasekaran2011rank} the non-zero entries of the sparse matrix are assumed to be drawn randomly following a uniform distribution. Following this line of research \cite{li2013compressed} apply these conditions to the matrix completion problem with randomly observed entries.  In a parallel work \cite{chen2013low} consider the case where the indices of the observed entries may be simultaneously both random and deterministic. In these papers only noiseless robust matrix completion is considered.

\cite{cambier2016robust} study computational aspects of robust matrix completion (in the previously mentioned setting). A method relying on Riemannian optimization is proposed. The authors assume the rank of the matrix to be estimated to be known.

In \cite{foygel2011learning} weighted nuclear norm penalized estimators with (possibly nonconvex) Lipschitz continuous loss functions are studied from a learning point of view. The partially observed entries are assumed to follow a possibly non-uniform distribution on the set $\chi$. In contrast, our derivations rely among other properties on the convexity of the risk (i.e. the margin conditions).

Noisy robust matrix completion was first investigated in \cite{klopp2014robust}. The authors assume that the truth $A^*$ is decomposed in a low-rank matrix and a sparse matrix where the low-rank matrix contains the ``parameters of interest'' and the sparse matrix contains the corruptions. In addition, every observation is corrupted by independent and centered subgaussian noise. The largest entries of both the low-rank and sparse matrices are assumed to be bounded (e.g. by the maximal possible rating). Their model is as follows: $(X_i,\tilde{Y}_i), i = 1, \dots N$ satisfy
\begin{equation}
\label{eqn:kloppmodel}
\tilde{Y}_i = \trace( X_i A^*) + \xi_i, i = 1, \dots, N,
\end{equation}
where $A^* = L^* + S^*$ with $L^*$ a low-rank matrix and $S^*$ a matrix with entrywise sparsity. Columnwise sparsity is also considered but in view of a comparison of this and our approach we prefer to restrict to entrywise sparsity. The masks $X_i$ are assumed to lie in the set $\chi$ (\ref{eqn:space}) and to be independent of the noise $\xi_i$ for all $i$. The set of observed indices is assumed to be the union of two disjoint components $\Xi$ and $\tilde{\Xi}$. The set $\Xi$ corresponds to the non-corrupted noisy observations (i.e. only entries of $L^*$ plus $\xi_i$). The entries corresponding to these observations of $S^*$ are zero. The set $\tilde{\Xi}$ contains the indices of the observations that are corrupted by a (nonzero) entry of $S^*$. It is not known if an observation comes from $\Xi$ or $\tilde{\Xi}$. The estimator given in \cite{klopp2014robust} is
\begin{small}
\begin{equation}
\label{eqn:kloppest}
(\hat{L}, \hat{S}) \in \underset{\substack{\Vert L \Vert_\infty \leq \eta \\  \Vert S \Vert_\infty \leq \eta}}{\arg \min} \left\lbrace \frac{1}{N} \sum_{i=1}^N (\tilde{Y}_i - \trace(X_i (L+S)) )^2 + \lambda_1 \Vert L \Vert_{\text{nuclear}} + \lambda_2 \Vert S \Vert_1 \right\rbrace.
\end{equation}
\end{small}

In contrast to the previously mentioned papers on robust matrix completion, we consider (possibly heavy-tailed) random errors that affect the observations but not the truth.

\subsection{Organization of the paper}

The paper is organized in the following way. We state the main assumptions that are used throughout  the paper in Section \ref{s:nuclear}. Then, the nuclear norm, its properties, and its similarities to the $\ell_1$-norm are discussed. To bound the empirical process part resulting from the matrix completion setting we make use of the results in Section \ref{s:empiric} of the Supplementary Material (Appendix \ref{s:supplement}). In Section \ref{s:oracle} the main theorems are presented: the (deterministic) sharp and non-sharp oracle inequalities. In Section \ref{ss:lowrankoracle} we present the applications of these results to the case of Huber loss and absolute value loss. The asymptotics and the applications to weak sparsity are presented in Section \ref{s:asymptotics}. Finally, to verify the theoretical findings, Section \ref{s:simulations} presents some simulations. The Student t distribution with three degrees of freedom is considered as an error distribution. \smallskip

\section{Preliminaries}
\label{s:nuclear}
In this section the assumptions on the loss functions, the risk, and the distribution of the errors are presented. In particular, Assumptions 1-3 below are on the curvature of the (theoretical) risk. They are used to derive the \emph{deterministic} sharp and non-sharp oracle inequalities. It is important to notice that the curvature of the risk mainly depends on the properties of the distribution of the errors. Assumptions 4 and 5 below will be shown to be sufficient for Assumptions 2 and 3 to hold, respectively.

Furthermore, we also discuss the properties of the nuclear norm. Thanks to the penalization term in the objective functions the optimization problems become computationally tractable. We also highlight the commonalities of the vector $\ell_1$-norm and the nuclear norm for matrices.

\subsection{Assumptions on the risk and the distribution of the errors}
The first assumption is about the loss function.

\begin{assumption}
\label{ass:lipschitz}
Let $\rho$ be the loss function. We assume that it is Lipschitz continuous with constant $L$, i.e.  that for all $x,y \in \mathbb{R}$
\begin{equation}
\vert \rho(x)- \rho(y) \vert \leq L \vert x-y \vert.
\end{equation} 
\end{assumption}
The next two assumptions ensure the identifiability of the parameters by requiring a sufficient convexity of the risk around the target. 
\begin{assumption}
\label{ass:1margin}
One-point-margin condition. There is an increasing \linebreak strictly convex function $G$ with $G(0) =0$ such that for all $B \in \mathcal{B}$
\begin{equation*}
R\left(B\right) - R\left(B^0\right) \geq G\left( \Vert B-B^0 \Vert_F \right),
\end{equation*}
where $R$ is the theoretical risk function.
\end{assumption}
\begin{assumption}
\label{ass:2margin}
Two-point-margin condition. There is an increasing \linebreak strictly convex function $G$ with $G(0)=0$ such that for all $B, B' \in\mathcal{B}$ we have
\begin{equation*}
R\left(B \right) - R \left(B' \right) \geq \trace \left( \dot{R} \left(B' \right)^T \left(B-B' \right) \right) + G\left( \Vert B - B' \Vert_F \right),
\end{equation*}
where $R$ is the theoretical risk function and $[\dot{R}(B') ]_{kl} = \frac{\partial}{\partial B_{kl}} R(B) \vert_{B = B'}$.
\end{assumption}

Assumption \ref{ass:lipschitz} is crucial when it comes to the application of the Contraction Theorem which in turn allows us to apply the dual norm inequality to find a bound for the random part of the oracle bounds. Assumptions \ref{ass:1margin} and \ref{ass:2margin} are essential in the proofs of the (deterministic) results. In particular, in addition to the differentiability of the empirical risk $R_n$, Assumption \ref{ass:2margin} is responsible for the sharpness of the first oracle bound that will be proved. The margin conditions are strongly related to the shape of the distribution function and the corresponding density of the errors.

For the specific application to the Huber loss and absolute value loss estimators we show that mild conditions on the distribution of the errors ensure a sufficient curvature of the risk for both loss functions under study.

Assumption \ref{ass:2margin} holds under a weak condition on the distribution function of the errors:
\begin{assumption}
\label{ass:ass4}
Assume that there exists a constant $C_1>0$ such that the distribution function $F$ with density with respect to Lebesgue measure $f$ of the errors fulfills
\begin{equation}
\label{ass:distn}
F(u + \kappa) - F(u - \kappa) \geq 1/ C_1^2, \ \mbox{for all} \ \vert u \vert \leq 2 \eta \ \mbox{and} \ \kappa \leq 2 \eta.
\end{equation}
\end{assumption}

\begin{lemma}\label{lemma:2margin}
Assumption \ref{ass:ass4} implies Assumption \ref{ass:2margin} with $G(u) = u^2/(2C_1^2 pq)$.
\end{lemma}

The following assumption guarantees that Assumption \ref{ass:1margin} holds.
\begin{assumption}
\label{ass:ass5}
Suppose $\varepsilon_1, \dots, \varepsilon_n$ are i.i.d. with median zero and density $f$ with respect to Lebesgue measure. Assume that for $C_2 > 0$
\begin{equation}
f \left(u \right) \geq \frac{1}{C_2^2}, \ \mbox{for all} \ \left\vert u \right\vert \leq 2 \eta.
\end{equation}
\end{assumption}

\begin{lemma} \label{lemma:1margin}
Assumption \ref{ass:ass5} implies Assumption \ref{ass:1margin} with  $G(u) = u^2/(2C_2^2pq)$.
\end{lemma}

Another important fact is that when the distribution of the errors is assumed to be symmetric around zero $B^* = B^0$. This phenomenon is discussed in Section 4 of the Supplement.

\subsection{Properties of the nuclear norm}
The regularization by the nuclear norm plays a similar role as the $\ell_1$-norm in the Lasso \citep{tibshirani1996regression}. We illustrate the similarities and differences of these types of regularizations. In view of the oracle inequalities and in order to keep the notation as simple as possible we discuss the properties of the nuclear norm of the oracle. The oracle is typically a value $B$ that takes an up-to-constants optimal trade-off between approximation error and estimation error. In what follows, $B$ is called ``the oracle'' although its choice is flexible.

Consider the singular value decomposition of the oracle $B$ with rank $s^\star$ given by
\begin{equation}
\label{eqn:oracle}
B = P \Lambda Q^T,
\end{equation}
where $P$ is a $p \times q$ matrix, $Q$ a $q \times p$ matrix and $\Lambda$ a $q \times q$ diagonal matrix containing the ordered singular values $\Lambda_1  \geq \dots \geq \Lambda_q $.
Then the nuclear norm is given by
\begin{equation}
\left\Vert B \right\Vert_{\mbox{nuclear}} = \sum_{i = 1}^q \Lambda_i (B)  = \left\Vert \Lambda(B) \right\Vert_1,
\end{equation}
by interpreting $\Lambda(B) \in \mathbb{R}^q$ as the vector of singular values. The penalization with the nuclear norm induces sparsity in the singular values, whereas the penalization with the vector $\ell_1$-norm of the parameters in linear regression induces sparsity directly in the parameters. On the other hand, the rank plays the role of the number of non-zero coefficients in the Lasso setting, namely
\begin{equation}
s^\star = \left\Vert \Lambda(B) \right\Vert_0.
\end{equation}

One main ingredient of the proofs of the oracle inequalities is the so-called triangle property as introduced in \cite{van2001least}. This property was used in e.g. \cite{buhlmann2011statistics} to prove non-sharp oracle inequalities. For the $\ell_1$-norm the triangle property follows from its decomposability. For the nuclear norm the triangle property as it is used in this work depends on the features of the oracle $B$. For this reason we notice that for any positive integer $s \leq q$ the oracle can be decomposed in
\begin{equation*}
B = B^+ + B^-, \quad B^+ = \sum_{k=1}^s \Lambda_k P_k Q_k^T, \quad B^- = \sum_{k=s+1}^q \Lambda_k P_k Q_k^T.
\end{equation*}
The matrix $B^+$ is called ``active'' part of the oracle $B$, whereas the matrix $B^-$ is called the ``non-active'' part. The singular value decomposition of $B^+$ is given by
\begin{equation*}
B^+ = P^+ \Lambda Q^{+^T}.
\end{equation*}

We observe that the integer $s$ is not necessarily the rank of the oracle $B$. The choice of $s$ is free. One may choose a value that trades off the roles of the ``active'' part $B^+$ and ``non-active'' part $B^-$, see Lemma \ref{LEMMA:WEAKSPARS}. The following lemma is adapted from Lemma 7.2 and Lemma 12.5 in \cite{geer2015}.

\begin{lemma}
\label{def:triangle}
Let $B^+ \in \mathbb{R}^{p \times q}$ be the active part of the oracle $B$. Then we have for all $B' \in \mathbb{R}^{p \times q}$ with
\begin{align*}
\Omega^+_{B^+} \left(B'\right) &:= \sqrt{s} \left(\left\Vert P^+P^{+^T} B' \right\Vert_F + \left\Vert B' Q^+Q^{+^T} \right\Vert_F + \left\Vert P^+P^{+^T}B'Q^+Q^{+^T} \right\Vert_F \right) \\
\mbox{and} \\
 \Omega_{B^+}^- \left(B' \right) &:= \left\Vert \left(I - P^+P^{+^T} \right) B' \left( I - Q^+Q^{+^T} \right) \right\Vert_{ \mbox{\emph{nuclear}}}
\end{align*}
that
\begin{equation}
\label{def:triangleprop}
\left\Vert B^+ \right\Vert_{\mbox{ \emph{nuclear}}} - \left\Vert B' \right\Vert_{\mbox{\emph{nuclear}}} \leq \Omega_{B^+}^+ \left( B' - B^+ \right) - \Omega_{B^+}^- \left( B' \right).
\end{equation}
We then say that the triangle property holds at $B^+$.

In particular, since $\Omega_{B^+}^+(B^-) = 0$, we have for any $B' \in \mathbb{R}^{p \times q}$
\begin{equation}
\label{lemma:triangleprop}
\left\Vert B \right\Vert_{\mbox{\emph{nuclear}}}- \left\Vert B' \right\Vert_{\mbox{\emph{nuclear}}} \leq \Omega^+_{B^+}(B' - B) - \Omega^-_{B^+} (B'-B) + 2 \left\Vert B^- \right\Vert_{\mbox{\emph{nuclear}}}.
\end{equation}
Moreover, we have
\begin{equation}
\label{eqn:nucomegaplusomegaminus}
\left\Vert \cdot \right\Vert_{\mbox{\emph{nuclear}}} \leq \Omega_{B^+}^+ + \Omega_{B^+}^-.
\end{equation}
\end{lemma}
From now on, we write $\Omega^+ = \Omega_{B^+}^+$ and $\Omega^-=\Omega_{B^+}^- $. Equation \ref{lemma:triangleprop} is proved in Appendix \ref{appendix:proofmain}.


Hence, the property that our estimators should mimic is not the rank of the oracle but rather the fact that the ``non-active'' part is zero under the semi-norm induced by the active part.

Moreover, we define the norm $\underline{\Omega}$ as
\begin{equation*}
\underline{\Omega} := \Omega^+ +  \Omega^-.
\end{equation*}


\begin{remark}
Notice that the semi-norms $\Omega^+$ and $\Omega^-$ form a complete pair, meaning that $\underline{\Omega} := \Omega^+ + \Omega^-$ is a norm.
\end{remark}

The estimation error in several different norms can thus be ``computed'' in general (semi-)norms.

A tail bound for the maximal singular value of a finite sum of matrices lying in the set $\chi$ defined in Equation (\ref{eqn:space}) is given in the following theorem. For this purpose, we first need to define the Orlicz norm of a random variable.
Let $Z \in \mathbb{R}$ be a random variable and $\alpha \geq 1$ a constant. Then the $\Psi_\alpha-$Orlicz norm is defined as
\begin{equation}
\left\Vert Z \right\Vert_{\Psi_\alpha} := \inf \left\lbrace c > 0: \mathbb{E} \exp \left[\left\vert Z \right\vert^\alpha / c^\alpha \right] \leq 2 \right\rbrace
\end{equation}
\begin{theorem} [\begin{small} Proposition 2 in \cite{Koltchinskii2011} \end{small}]
\label{thm:orlicz}
Let $X_1, \dots, X_n$ be i.i.d. $q\times p$ matrices that satisfy for some $\alpha \geq 1$ (and all $i$)
\begin{equation*}
\mathbb{E} X_i=0, \left\Vert \Lambda_{\max} (X_i) \right\Vert_{\Psi_\alpha} =: K < \infty.
\end{equation*}
Define
\begin{equation*}
S^2 := \max \left\lbrace \Lambda_{\max} \left( \sum_{i =1}^n \mathbb{E} X_i X_i^T \right)/n, \Lambda_{\max} \left(\sum_{i=1}^n \mathbb{E} X_i^T X_i \right)/n \right\rbrace.
\end{equation*}
Then for a constant $C$ and for all $t>0$
\begin{align*}
\mathbb{P} \left( \Lambda_{\max} \left(\sum_{i=1}^n X_i \right)/n \geq C S \sqrt{\frac{t + \log(p+q)}{n}} \right. \\
\left. + C \log^{1/\alpha} \left( \frac{K}{S} \right) \left( \frac{t + \log(p+q)}{n} \right) \right) \leq \exp(-t).
\end{align*}
\end{theorem}
This theorem is used with the tail summation property of the expectation in the derivations of the tail bounds in Section 3 of the Supplementary Material.

\section{Oracle inequalities}
We first give two deterministic sharp and non-sharp oracle inequalities. The connection to the empirical process parts and to the specific loss functions follows in Subsection \ref{ss:lowrankoracle}. Let $B^0 = \underset{B' \in \mathcal{B}}{\argmin} \ R(B')$ be the target.
It is assumed that $q \leq p $.

\label{s:oracle}
\subsection{Sharp oracle inequality}
Here, we assume that the loss function is differentiable and Lipschitz continuous. The next lemma gives a connection between the empirical risk and the penalization term.

\begin{lemma}[adapted from Lemma 7.1 in \cite{geer2015}]
\label{lemma:twopoint}
Suppose that $R_n$ is differentiable. Then for all $B \in \mathcal{B}$
\begin{equation*}
-\trace \left( \dot{R}_n(\hat{B})^T(B-\hat{B }) \right) \leq \lambda \Vert B \Vert_{\mbox{\emph{nuclear}}}- \lambda \Vert \hat{B} \Vert_{\mbox{\emph{nuclear}}}.
\end{equation*}
\end{lemma}

The following theorem is inspired by Theorem 7.1 in \cite{geer2015}. In contrast to this theorem, we need to bound the empirical process part differently. In view of the application to the matrix completion problem, we assume a specific bound on the empirical process.

\begin{theorem}
\label{thm:sharp}
Suppose that Assumptions \ref{ass:lipschitz} and \ref{ass:2margin} hold, that the loss function is differentiable, and let $H$ be the convex conjugate of $G$. Assume further for all $B' \in \mathcal{B}$ that for $\lambda_\varepsilon > 0$ and $\lambda_* >0$
\begin{equation*}
\label{eqn:empbound1}
\left\vert \trace \left( \left(\dot{R}_n \left( B' \right) - \dot{R} \left( B' \right)\right)^T \left(B - B' \right) \right) \right\vert \leq \lambda_\varepsilon\underline{\Omega}(B'-B) +\lambda_*.
\end{equation*}
 Take $\lambda > \lambda_\varepsilon$. Let $0 \leq \delta < 1$ be arbitrary, and define
\begin{equation*}
\underline{\lambda} := \lambda -\lambda_\varepsilon, \quad \overline{\lambda} := \lambda_\varepsilon+ \lambda + \delta \underline{\lambda} 
\end{equation*}
Then
\begin{align*}
&\delta \underline{\lambda} \Omega^+(\hat{B} - B) +  \delta \underline{\lambda} \Omega^- ( \hat{B} - B) + R ( \hat{B} ) - R(B) \nonumber \\
&\leq  H(\overline{\lambda}3 \sqrt{s}) + 2 \lambda \left\Vert  B^- \right\Vert_{\mbox{\emph{nuclear}}} + \lambda_*.
\end{align*}
\end{theorem}
In the proof of this theorem the differentiability of the loss function and Assumption \ref{ass:2margin} are crucial. Without this property an additional term arising from the one-point-margin condition would appear in the upper bound. This term would then lead to a non-sharp bound.

\subsection{Non-sharp oracle inequality}
Instead of bounding an empirical process term depending on the derivative of the empirical and theoretical risks we need to consider differences of these functions.

\begin{theorem}
\label{thm:nonsharp}
Suppose that Assumptions \ref{ass:lipschitz} and \ref{ass:1margin} hold. Let $H$ be the convex conjugate of $G$. Suppose further that for $\lambda_\varepsilon >0$, $\lambda_*> 0$, and all $B' \in \mathcal{B}$
\begin{equation}
\label{eqn:assbound}
 \left\vert \left[R_n (B') - R(B') \right] - \left[R_n (B) - R(B) \right] \right\vert \leq \lambda_\varepsilon\underline{\Omega} (B' - B)  + \lambda_*.
\end{equation}
Let $0 < \delta <1$, take $\lambda > \lambda_\varepsilon $ and define 
\begin{equation}
\overline{\lambda} = \lambda + \lambda_\varepsilon, \quad \underline{\lambda} = \lambda - \lambda_\varepsilon.
\end{equation}
Then,
\begin{align*}
	\delta \underline{\lambda} \underline{\Omega} (\hat{B} - B) &\leq 2H(\overline{\lambda} (1+\delta) 3\sqrt{s}) \\
	&\phantom{\dots} + 2\left( \lambda_* + (R(B) - R(B^0)) \right)+ 4 \lambda \Vert B^- \Vert_{\mbox{\emph{nuclear}}}
\end{align*}
	and
\begin{align*}
R(\hat{B}) - R(B) &\leq \frac{1}{\delta} \left[ 2H(\overline{\lambda} (1+\delta) 3\sqrt{s}) +\lambda_* + 2 (R(B)-R(B^0)) \right. \\
&\phantom{\leq}\left. + 2 \lambda \Vert B^- \Vert_{\mbox{\emph{nuclear}}} \right] + \lambda_* + 2 \lambda \Vert B^- \Vert_{\mbox{\emph{nuclear}}}.
\end{align*}
\end{theorem}
It has to be noticed that the above bound is ``good'' only if $R(B) - R(B^0)$ is already small. The main cause for the non-sharpness is Assumption \ref{ass:1margin} that leads to an additional term in the upper bound of the inequality.
\subsection{Applications to specific loss functions}
\label{ss:lowrankoracle}
We now apply the deterministic sharp and non-sharp oracle inequalities to the case of the Huber loss and absolute value loss, respectively. We assume in both cases that the distribution of the errors is symmetric around $0$ so that $B^0 = B^*$. This is discussed in detail in Section \ref{ss:distoferrors}.
\bigskip

\textbf{Huber Loss-sharp oracle inequality}

We first consider the case that arises by choosing the Huber loss. Theorem \ref{thm:sharp} together with Lemma \ref{lemma:2margin} and the first claim of Lemma 3.2 in the Supplement imply the following corollary. It is useful to notice that the Lipschitz constant of the Huber loss is $2 \kappa$.

\begin{corollary}
\label{cor:usrcompletionhuber}
Let $B = B^+ + B^-$ where $B^+$ and $B^-$ are defined in Equation \ref{eqn:oracle}. Let Assumption \ref{ass:ass4} be satisfied.

 For a constant $C_0 > 0$ let 
\begin{align*}
\lambda_\varepsilon= 2 (4\eta + 2 \kappa) \left( (8 C_0 + \sqrt{2})\sqrt{\frac{\log (p+q)}{nq}} + 8C_0\sqrt{\log(1+q)} \frac{\log(p+q)}{n} \right)
\end{align*}
and $\lambda_* = 8 \eta (4\eta + 2 \kappa) p \log(p+q)/(3n) + \lambda_\varepsilon$.

Assume that $\lambda>\lambda_\varepsilon$. Take $0 \leq \delta < 1$,
\begin{equation}
\underline{\lambda} := \lambda - \lambda_\varepsilon \ \mbox{and} \ \overline{\lambda} := \lambda_\varepsilon + \lambda + \delta \underline{\lambda}
\end{equation}
Choose $j_0 := \lceil \log_2 (7 q \sqrt{pq} \eta )\rceil$ and define
\begin{equation*}
\alpha = (j_0 + 2) \exp(-p \log(p+q)).
\end{equation*}
Then we have with probability at least $1-\alpha$ that
\begin{align*}
&\delta \underline{\lambda} \Omega^+( \hat{B}_H - B ) + \delta \underline{\lambda} \Omega^- (\hat{B}_H - B) + R (\hat{B}_H) - R(B) \\
&\leq \frac{p q C_1^2 \bar{\lambda}^2 9s}{2} + 2 \lambda \left\Vert B^- \right\Vert_{\mbox{\emph{nuclear}}} + \lambda_*.
\end{align*} 
\end{corollary}
Assumption \ref{ass:ass4} guarantees that the risk function is sufficiently convex. From this assumption we also obtain a bound for the possible values of the tuning parameter $\kappa$. We can also see that the results hold for errors with a heavier tail than the Gaussian. The choice of the noise level $\lambda_\varepsilon$ and consequently of the tuning parameter $\lambda$ results from the the probability inequalities for the empirical process in Section \ref{s:empiric}.  The quantity $\lambda_*$ is also a consequence of the bound on the empirical process part. However, it does not affect the asymptotic rates. 

\bigskip
\textbf{Absolute value loss - non-sharp oracle inequality}

The next corollary is an application to the case of the absolute value loss. Theorem \ref{thm:nonsharp} combined with Lemma \ref{lemma:1margin} and the second claim of Lemma 3.2 in the Supplement lead to the following corollary. The Lipschitz constant in this case is $1$.
\begin{corollary}
\label{thm:oracle}
Let the oracle $B$ be as in \ref{eqn:oracle}. Suppose that Assumption \ref{ass:ass5} is satisfied.
 For a constant $C_0 > 0$ let
 \begin{align*}
 \lambda_\varepsilon =2 \left( (8 C_0 + \sqrt{2})\sqrt{\frac{\log (p+q)}{nq}} + 8C_0\sqrt{\log(1+q)} \frac{\log(p+q)}{n} \right)
 \end{align*}
 and  $\lambda_* = 8 \eta p \log(p+q)/(3n) + \lambda_\varepsilon$.
Take $0< \delta< 1$ and $\lambda > \lambda_\varepsilon$. Choose $j_0 := \lceil \log_2 (7 q \sqrt{pq} \eta )\rceil$ and define
\begin{equation*}
\alpha = (j_0 + 2) \exp(-p \log(p+q)).
\end{equation*}
Then we have with probability at least  $1-\alpha$ that 
\begin{align*}
&\delta \underline{\lambda} \underline{\Omega} (\hat{B} - B) \\
&\leq 6 C^2 \overline{\lambda}^2 \left(1+\delta \right)^2 p q s + 2 \lambda_* + 2(R(B) - R(B^0)) + 4 \lambda \Vert B^- \Vert_{\mbox{\emph{nuclear}}}
\end{align*}
and 
\begin{align*}
R(\hat{B}) - R(B) &\leq \frac{1}{\delta} \Big[  6 C_2^2 \overline{\lambda}^2 \left(1+\delta \right)^2 p q s + \lambda_* + 2 (R(B)-R(B^0)) \\
&\phantom{\leq \frac{1}{\delta}}+ 2 \lambda \Vert B^- \Vert_{\mbox{\emph{nuclear}}} \Big] + \lambda_* + 2 \lambda \Vert B^- \Vert_{\mbox{\emph{nuclear}}}.
\end{align*}
\end{corollary}
Also in this case, the choices of $\lambda_\varepsilon$ and $\lambda_*$ are a consequence of the probability bounds.
\section{Asymptotics and Weak sparsity}
\label{s:asymptotics}
The results in Section \ref{s:oracle} are valid for finite values of the dimension of the matrix $p,q$, the rank, and the number of observed entries $n$. A question that is answered in this section is how the estimation errors of the proposed estimators behave when $n, p,$ and $q$ are allowed to grow.

As mentioned in \cite{negahban2012restricted}, practical reasons motivate the assumption that the matrix $B^0$ is not exactly low-rank but only approximately. In relation to the matrix completion problem one observes that the ratings given by the users are unlikely to be exactly equal but rather very similar. This translates to a matrix that is not low-rank. However, it is sensible to assume that the matrix is almost low-rank. The notion of weak sparsity quantifies this assumption by assuming that for some $0 <r<1$ and $\rho > 0$
\begin{equation}
\label{eqn:weaksparsity}
 \sum_{k=1}^q \left( \Lambda_k^0 \right)^r =: \rho_r^r ,
\end{equation}
where $\Lambda_1^0, \dots, \Lambda_q^0$ are the singular values of $B^0$. For $r = 0$ we have under the convention that $0^0 = 0$ that
\begin{equation*}
\sum_{k=1}^q \left( \Lambda_k^0 \right)^0 = \sum_{k=1}^q \mathbbm{1}_{\left\lbrace \Lambda_k^0 > 0 \right\rbrace} = s_0,
\end{equation*}
where $s_0$ is the rank of $B^0$. The following lemma gives a bound of the non-active part of the matrix $B$ that appears in the oracle bounds.

\begin{lemma}
\label{lemma:weakspars}
\label{LEMMA:WEAKSPARS}
For $\sigma > 0$ we may take
\begin{equation}
\left\Vert B^- \right\Vert_{\mbox{\emph{nuclear}}} \leq \sigma^{1-r} \rho_r^r,
\end{equation}
and
\begin{equation*}
s \leq \sigma^{-r} \rho_r^r.
\end{equation*}
\end{lemma}
We first consider the asymptotic behavior of our estimators in the case of an exactly low-rank matrix and deduce from this the asymptotics for the case of an approximately low-rank matrix.

\subsection{Asymptotics}
\label{ss:asymptotics}
\subsubsection{sharp}
\label{sss:asymptoticssharp}
By Corollary \ref{cor:usrcompletionhuber}, assuming that $q \log(1+q) = o \left( \frac{n}{\log(p+q)} \right)$, and therefore using the choice for the noise level
\begin{equation*}
\lambda_\varepsilon \asymp \sqrt{\frac{\log(p+q)}{nq}}
\end{equation*}
we obtain
\begin{align}
R(\hat{B}_H) - R(B^0) &\leq R(B) - R(B^0)  \nonumber \\
&\phantom{\dots} + \mathcal{O}_{\mathbb{P}} \left(\frac{p s \log(p+q)}{n} +\sqrt{\frac{\log (p+q)}{n q}} \left(1+  \left\Vert B^- \right\Vert_{\mbox{nuclear}} \right)\right).\label{eqn:rate1}
\end{align}
We choose for simplicity the oracle to be the matrix $B^0$ itself with $s_0 = \mbox{rank}(B^0)$. Then, we make use of the two point margin condition that is shown to hold in Lemma \ref{lemma:2margin}.The resulting rate is then given by
\begin{equation}
\label{eqn:sharprate}
\Vert \hat{B}_H - B^0 \Vert_F^2 = \mathcal{O}_{\mathbb{P}} \left( (4\eta+2\kappa)^2 C_1^4 \frac{p^2 q s_0 \log(p+q)}{n} \right),
	\end{equation}
	where $\kappa$ is the Huber parameter and $C_1$ is the constant from Lemma \ref{lemma:2margin}.

\begin{remark}
The rate (\ref{eqn:sharprate}) depends on $\eta$ as in \cite{Koltchinskii2011} and on the Lipschitz constant of the loss function which is typically smaller than $\eta$. If $C_1^2 = O(\eta)$, the constant in front of the rate is of order $O(\eta^4)$. This is a ``worst-case'' scenario that shows the cost that is paid when allowing for very general error distributions as in our case. We emphasize that in this case the distribution of the errors is not required to have a density.
\end{remark}

In addition to the rate obtained for the Frobenius norm, we are also able to derive rates for the estimation error measured in nuclear norm. From Corollary \ref{cor:usrcompletionhuber} and Equation  \ref{eqn:nucomegaplusomegaminus} under the previous conditions it follows that
\begin{equation}
\Vert \hat{B}_H - B^0\Vert_{\text{nuclear}} = \mathcal{O}_{\mathbb{P}} \left(C_1^2 p q s_0 \sqrt{\frac{\log(p + q)}{nq}} \right).
\end{equation}

\subsubsection{non-sharp}
\label{sss:asymptoticsnonsharp}
By Corollary \ref{thm:oracle} it is known that the assumption \linebreak $q \log(1+q) = o \left( \frac{n}{\log(p+q)} \right)$ leads to the choice $\lambda_\varepsilon \asymp \sqrt{\log (p+q) /nq}$. Therefore, we have
\begin{align}
&R ( \hat{B} ) - R ( B^0 ) \nonumber \\
&= \mathcal{O}_\mathbb{P} \left( \frac{p s \log (p+q)}{n} + R \left(B \right) - R \left(B^0 \right) + \sqrt{\frac{\log (p+q)}{nq}} (1+\left\Vert B^- \right\Vert_{\mbox{nuclear}}) \right). \label{eqn:rate2}
\end{align}
What can be observed comparing the rates in Equations (\ref{eqn:rate1}) and (\ref{eqn:rate2}) is the presence of the additional term $R(B) - R(B^0)$ in the non-sharp case in contrast to the sharp case. We choose again the oracle to be the matrix $B^0$ itself. By the one point margin condition derived in Lemma \ref{lemma:1margin} we see that the rate of convergence in this case is given by
\begin{equation}
\Vert \hat{B} - B^0 \Vert_F^2 = \mathcal{O}_{\mathbb{P}} \left(C_2^4 \frac{p^2 q s_0 \log(p+q)}{n} \right),
\end{equation}
where the constant $C_2$ comes from Lemma \ref{lemma:1margin}.

\begin{remark}
If $C_2^2 = O(\eta)$  a comparison with the rates obtained in \cite{Koltchinskii2011} shows that the rates agree. In contrast to the rate obatined for the Huber loss (Equation \ref{eqn:sharprate}), the distribution of the errors is assumed to have a density. This leads to a constant of order $O(\eta^2)$ in a ``worst-case'' scenario. It is a natural consequence of the stronger assumption on the distribution of the errors. This is comparable to the constant obtained in \cite{Koltchinskii2011}.
\end{remark}

In analogy to the previous case, we are able to derive a rate for the estimation error measured in nuclear norm:
\begin{equation}
\Vert \hat{B} - B^0 \Vert_{\text{nuclear}} = \mathcal{O}_\mathbb{P} \left(C_2^2 p q s_0 \sqrt{\frac{\log (p+ q)}{nq}} \right).
\end{equation}

The rates are indeed very slow but this is not surprising given that per entry the number of observations is about $n/(pq)$. The price to pay for the estimation of the reduced number of parameters $p s_0$ is given by the term $\log(p+q)$.

\subsection{Weak sparsity}
\label{ss:weaksparsity}
In what follows, the asymptotic behavior of the proposed estimators is discussed when applied to an estimation problem where one aims at estimating a matrix that is not exactly low-rank.
With Lemma \ref{lemma:weakspars} and the rates given in the previous section we are able to derive an explicit rate also for the approximatley low-rank case. For this purpose, we assume that Equation \ref{eqn:weaksparsity} holds.

\subsubsection{Huber estimator}
\label{sss:weaksparsitysharp}
The following corollary gives rates for the estimation error of the Huber estimator when used for estimation of a not exactly low-rank matrix.
\begin{corollary}
\label{cor:rate}
With $q \log(1+q) =  o\left( \frac{n}{\log(p+q)}\right)$ we choose
\begin{equation*}
\lambda_\varepsilon \asymp \sqrt{\frac{\log (p+q)}{nq}}.
\end{equation*}
We then have
\begin{equation*}
\left\Vert \hat{B}_H - B^0 \right\Vert_F^2 = \mathcal{O}_{\mathbb{P}} \left((\eta + \kappa)^2 C_1^4 \frac{p^2 q \log(p+q)}{n} \right)^{1-r} \rho_r^{r}.
\end{equation*}
\end{corollary}

\subsubsection{Absolute value estimator}
\label{sss:weaksparsitynonsharp}
Using the oracle inequality under the weak sparsity assumption we obtain the following result.
\begin{corollary}
\label{cor:consist}
With $q \log(1+q) =  o\left( \frac{n}{\log(p+q)}\right)$ we choose
\begin{align*}
\lambda_\varepsilon &\asymp \sqrt{\frac{\log (p+q)}{nq}}.
\end{align*}
Then we have for the Frobenius norm of the estimation error
\begin{equation}
\left\Vert \hat{B} - B^0 \right\Vert_F^2 = \mathcal{O}_{\mathbb{P}} \left( \frac{p^2 q \log (p+q)}{n} \right)^{1-r} \rho_r^{r}.
\end{equation}
\end{corollary}

\section{Simulations}
\label{s:simulations}
In this section, the robustness of the Huber estimator \ref{eqn:huber} is empirically demonstrated. In Subsection \ref{ss:lowranksparse}, the Huber estimator is compared with the estimator proposed in \cite{klopp2014robust} under model \ref{eqn:traceregr} and \ref{eqn:kloppmodel} with each Student t and standard Gaussian noise. The sample size ranges in all simulations for all dimensions considered here from $3p\log(p)s_0$ to $pq$. Between minimal and maximal sample size there are in each case $10$ points.
To illustrate the rate derived in Section \ref{s:asymptotics} we compute the error $\Vert \hat{B}_H - B^0 \Vert_F^2$ for different dimensions of the problem under increasing number of observations.

To compute the solution of the optimization problem \ref{eqn:huber} functions from the Matlab library \texttt{cvx} \citep{cvx} were used.

Throughout this section the error is assumed to have the following shape
\begin{equation}
\mbox{Error} = \frac{1}{pq} \left\Vert \hat{B}_H - B^0 \right\Vert_F^2= \frac{1}{pq} \sum_{i = 1}^p \sum_{j = 1}^q \left(\hat{B}_{H_{ij}} - B_{ij}^0 \right)^2.
\end{equation}
To verify the robustness of the estimator \ref{eqn:huber} and the rate of convergence that was derived in Section \ref{s:asymptotics} we use the Student t distribution with $3$ degrees of freedom. Every point in the plots corresponds to an average of $25$ simulations. The value of the tuning parameter is set to 
\begin{equation*}
\lambda = 2 \sqrt{\frac{\log(p+q)}{nq}}.
\end{equation*}
A comparison with $\lambda_\varepsilon$ from Corollary \ref{cor:usrcompletionhuber} indicates that $\lambda$ is rather small. For the settings we consider in this section we found that this value for $\lambda$ is more appropriate. As done in \cite{candes2010matrix}, for a better comparison between the error curves of our estimator and the oracle rate in Equation (\ref{eqn:sharprate}) this rate was multiplied with $1.68$ in the case of Student t distributed errors and with $1.1$ in the case of Gaussian errors.

\subsection{t-distributed and gaussian errors}
\label{ss:student}
The variance of the Student t distribution with $\nu>2$ degrees of freedom is given by
\begin{equation}
\Var \left( \varepsilon_i \right) = \frac{\nu}{\nu -2}, \ \mbox{for} \ \varepsilon_i \sim t_\nu.
\end{equation}
Figure \ref{fig:comparsionlseandrobust} shows a comparison between the Huber estimator \ref{eqn:huber} with the estimator that uses the quadratic loss in the case of Student t with $3$ degrees of freedom distributed errors. As expected, the estimator that uses the quadratic loss is not robust against the corrupted entries. On the other hand, we can see in Figure \ref{fig:huberlsegaussian} that the Huber estimator performs almost as well as the quadratic loss estimator in the case of Gaussian errors with variance $1$. In agreement with the theory, the rate of the estimator is very close to the oracle rate for sufficiently large sample sizes. The value of $\kappa$ that we used in the simulations is $1.345$. The maximal rating $\eta$ is chosen to be $\eta = 10$.

\begin{figure}[h]
\hspace{-1.5cm}
\centering
\begin{subfigure}{0.45\textwidth}
\centering
\includegraphics[scale=0.35]{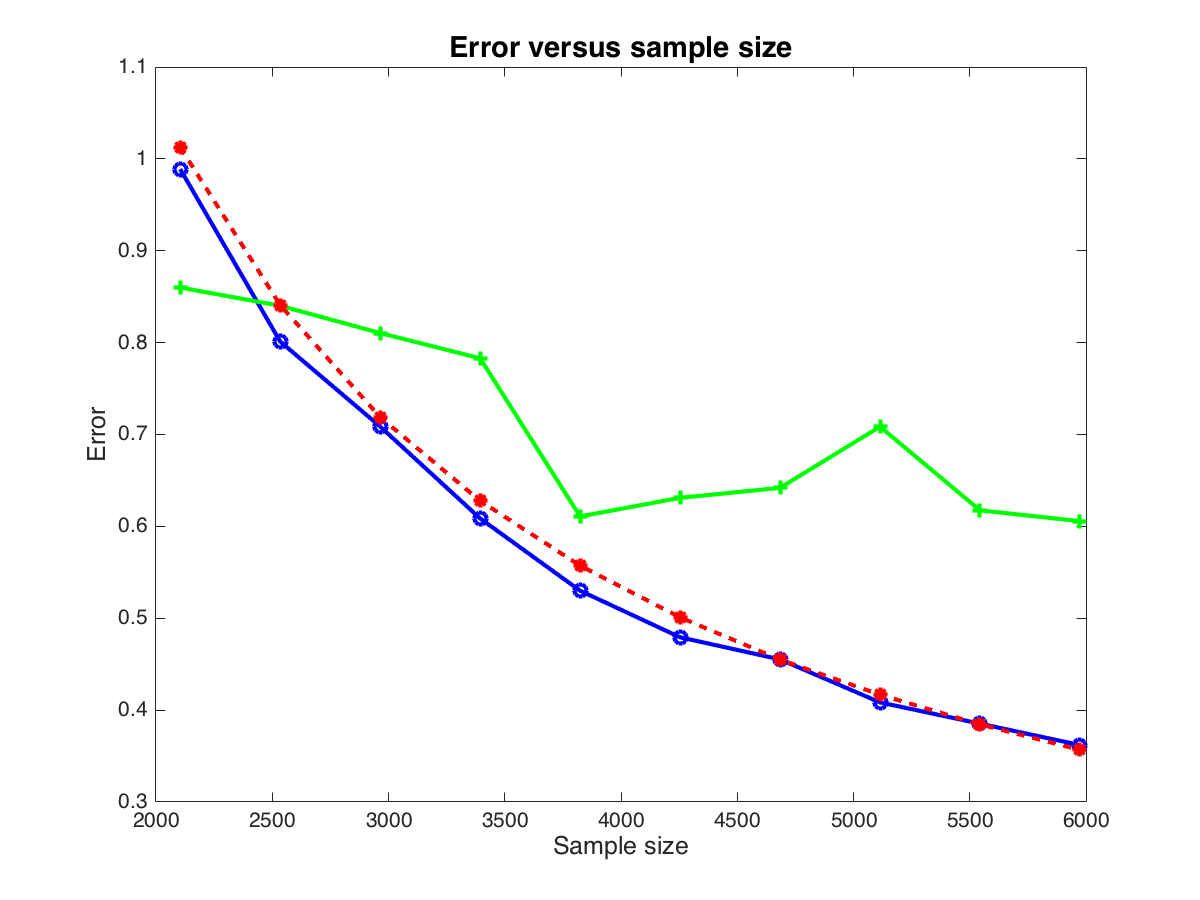}
\caption{}
\label{fig:comparsionlseandrobust}
\end{subfigure}
\hspace{0.85cm}
\begin{subfigure}{0.45\textwidth}
\centering
\includegraphics[scale=0.35]{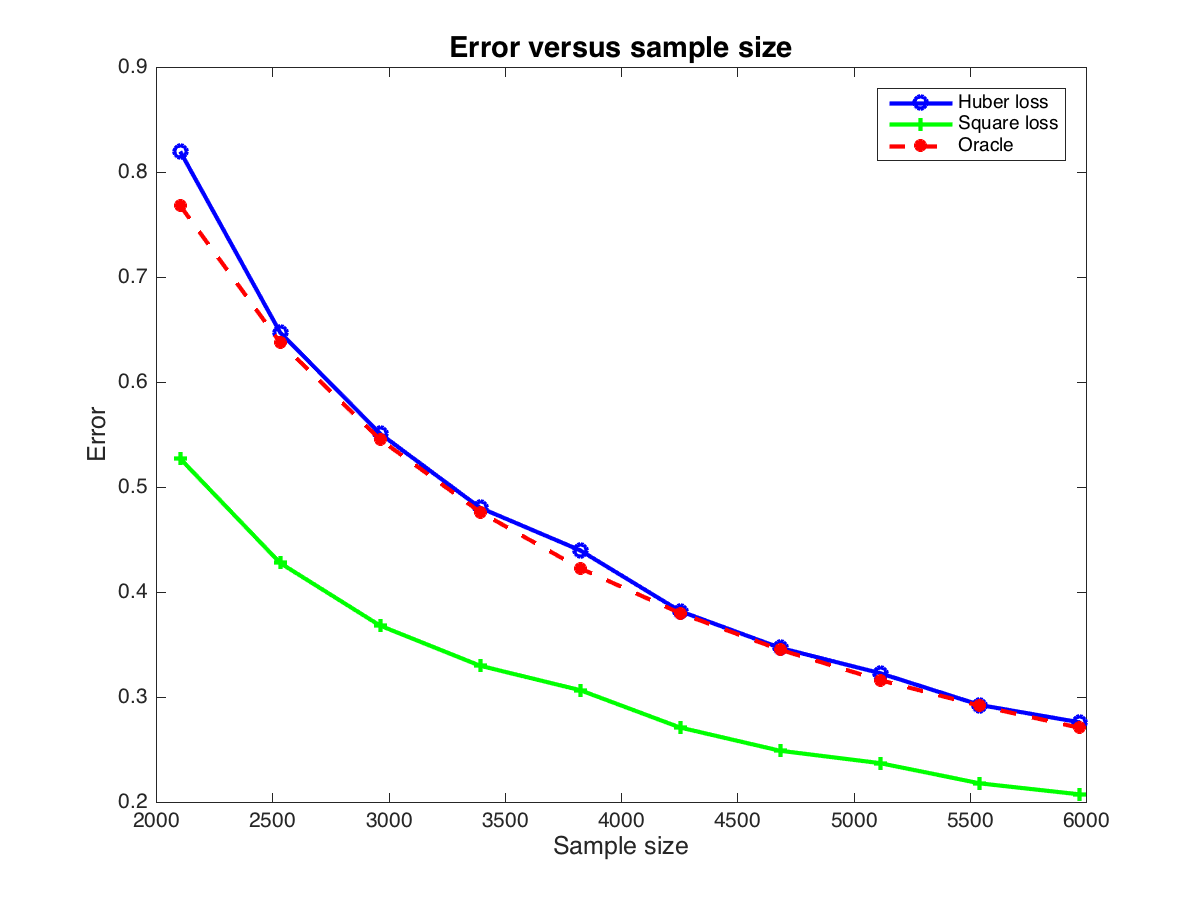}
\caption{}
\label{fig:huberlsegaussian}
\end{subfigure}
\caption{Comparison of the Huber and quadratic loss in the case of Student t distributed (left figure) and Gaussian errors (right figure). We choose $p=q=80$ and $s_0 = 2$. The dot-dashed red line corresponds to $1.68$ multiplied with the oracle bound derived in Equation \ref{eqn:sharprate} for the exact low-rank case. }
\end{figure}

\subsection{Changing the problem size}
\label{ss:problemsize}
In order to confirm/verify the theoretical results, we proceed similarly to what was done in \cite{negahban2011estimation} and \cite{negahban2012restricted} in the corresponding cases. Here, we consider three different problem sizes: $p,q \in \left\lbrace 30,50,80 \right\rbrace$. In Figure \ref{fig:threecases1} we observe that as the problem gets harder, i.e. as the dimension of the matrix increases, also the sample size needs to be larger. Figure \ref{fig:threecases2} shows that by rescaling the sample size by $n/(3ps_0\log(p))$ the rate of convergence agrees very well with the theoretical one. It is assumed that the rank of the matrices is $s_0=2$ for all cases. Every point corresponds to an average of $25$ simulations. The maximal rating $\eta$ and the tuning parameter $\kappa$ are chosen as before.

\begin{figure}[h]
\hspace{-1.5cm}
\centering
\begin{subfigure}{0.45\textwidth}
\centering
\includegraphics[scale=0.35]{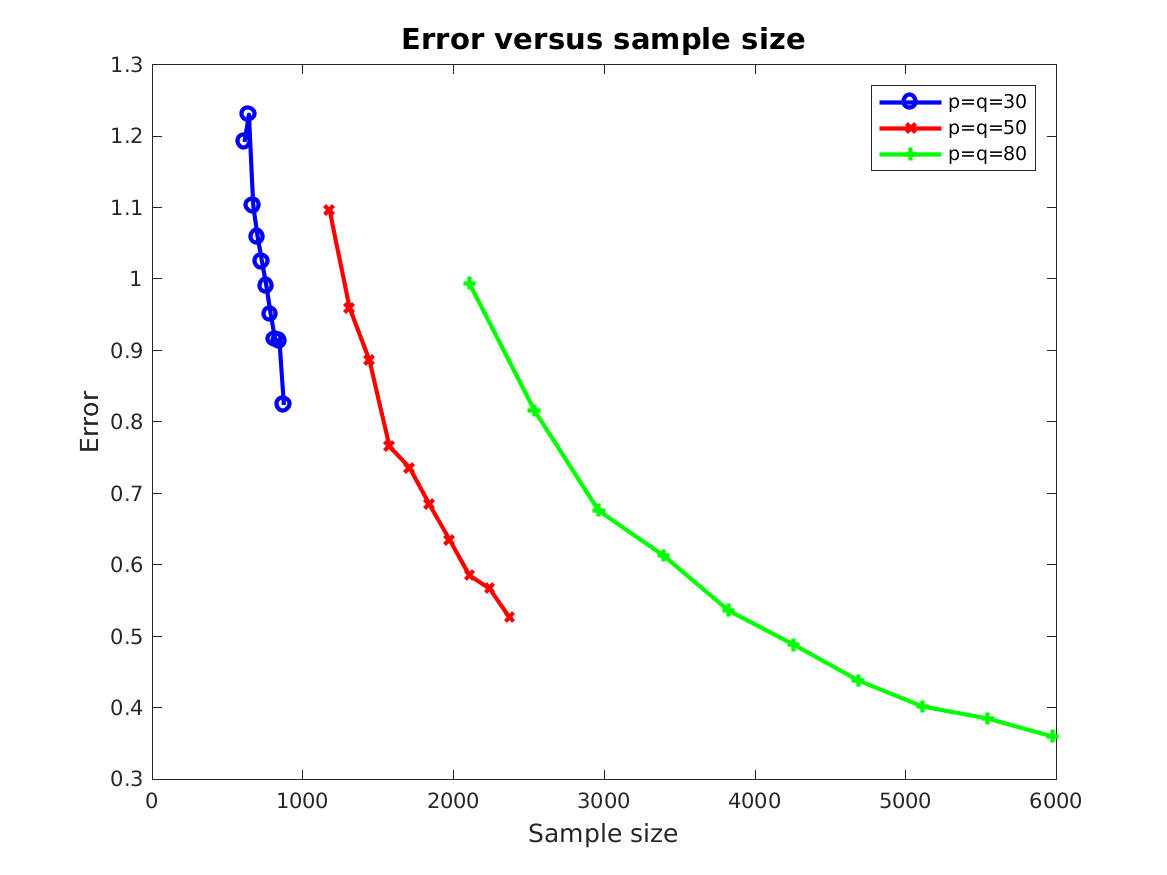}
\caption{}
\label{fig:threecases1}
\end{subfigure}
\hspace{0.85cm}
\begin{subfigure}{0.45\textwidth}
\centering
\includegraphics[scale=0.35]{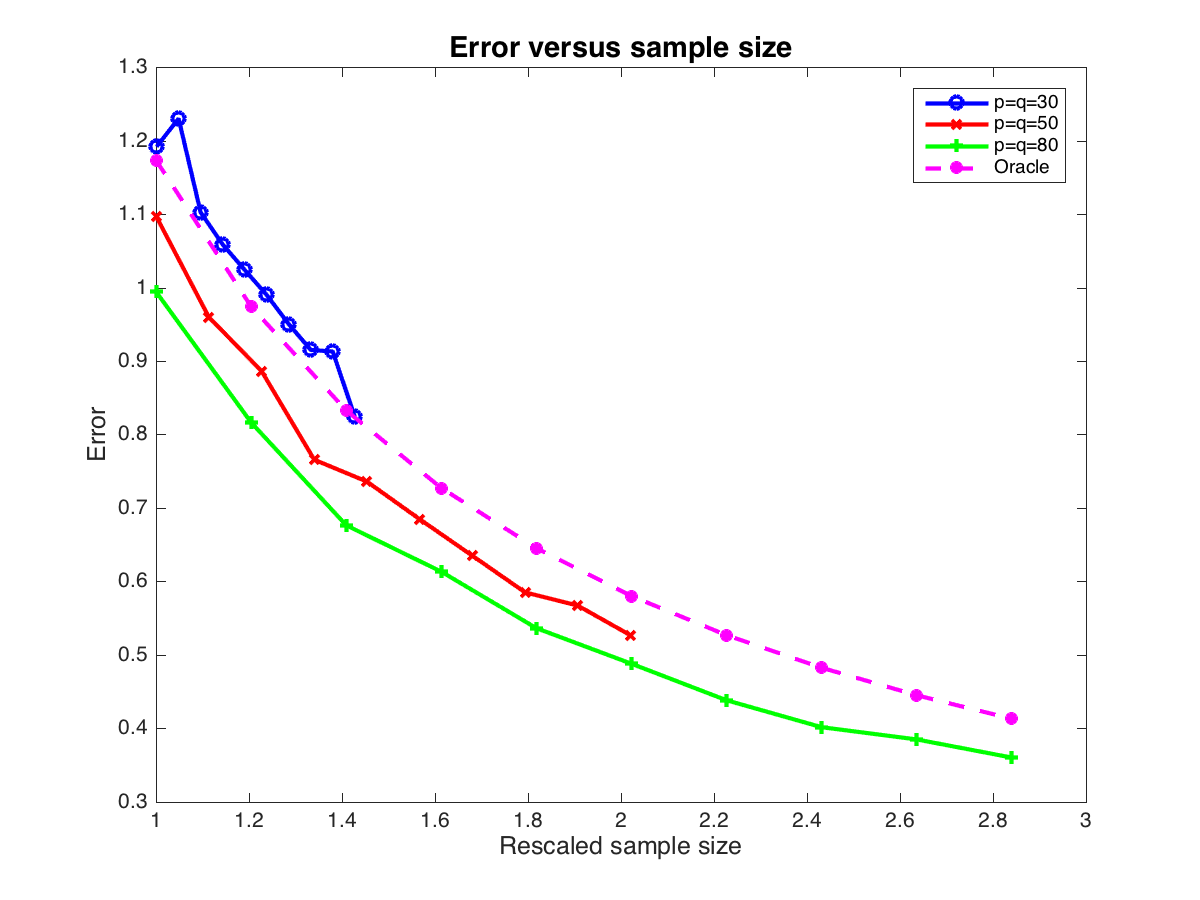}
\caption{}
\label{fig:threecases2}
\end{subfigure}
\caption{Three different problem sizes are considered: $p=q \in \left\lbrace 30,50, 80 \right\rbrace$. The rank is fixed to $s_0=2$ in both cases. In Panel (b) it can be seen that the rate of convergence corresponds approximately to the theoretical one derived in Equation \ref{eqn:sharprate}. The dot-dashed line is the oracle. It was multiplied by $1.68$ in order to fit our curves.}
\end{figure}

\subsection{Comparison with a low-rank + sparse estimator}
\label{ss:lowranksparse}
In this subsection, we compare the performance of the Huber estimator \ref{eqn:huber} with the performance of the low-rank matrix estimator proposed by \cite{klopp2014robust} \ref{eqn:kloppest}. We first compare the estimators $\hat{B}_H$ and $\hat{L}$ with the observations $Y_i$ generated according to the model \ref{eqn:traceregr} with standard Gaussian and Student t with $3$ degrees of freedom distributed errors. Equation (21) in \cite{klopp2014robust} suggests that the tuning parameters are chosen as follows:
\begin{equation*}
\lambda_1 =2 \sqrt{\frac{\log(p+q)}{nq}}, \quad \lambda_2 = 2 \frac{\log(p+q)}{n},
\end{equation*}
where $\lambda_1$ and $\lambda_2$ are the tuning parameters of the estimator \ref{eqn:kloppest}. Also in this case it has to be noticed that the tuning parameters are smaller than the theoretical values given in their paper.
\begin{figure}[h]
\hspace{-1.5cm}
\centering
\begin{subfigure}{0.45\textwidth}
\centering
\includegraphics[scale=0.35]{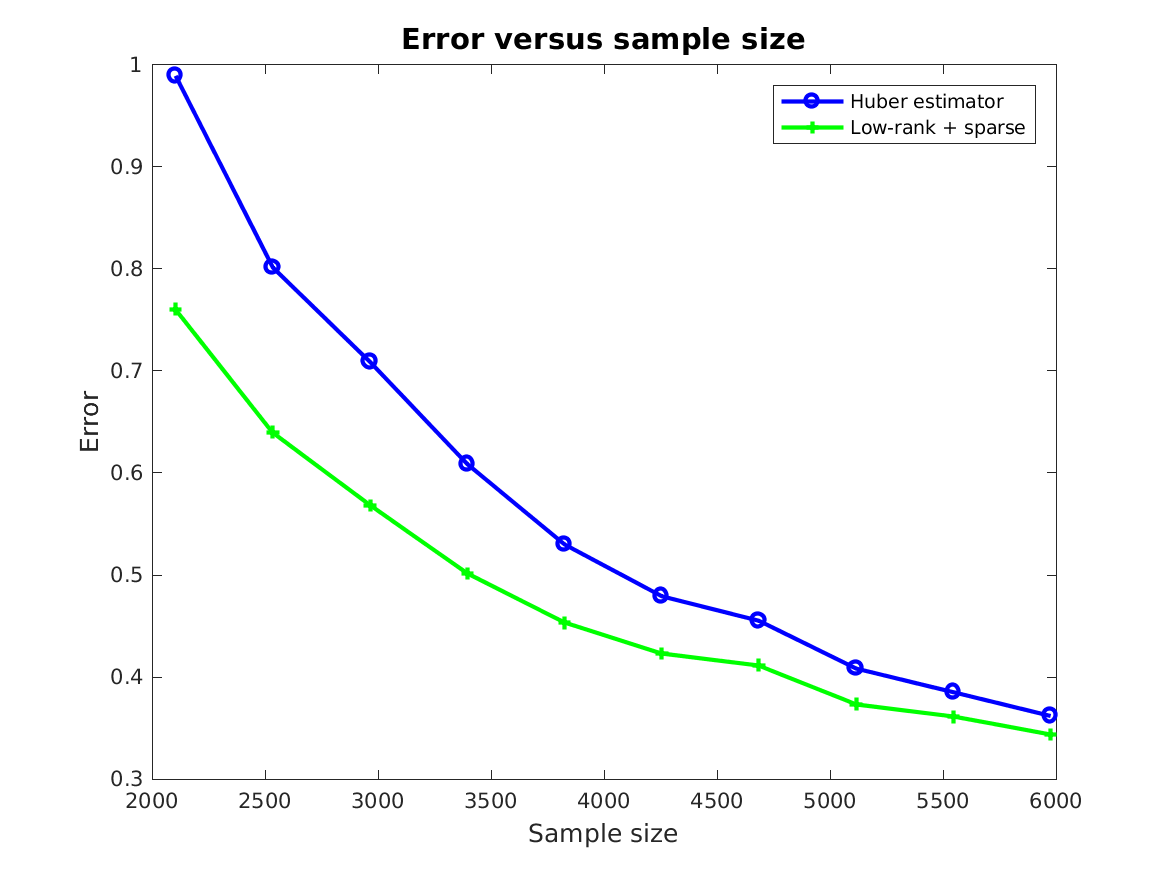}
\caption{}
\label{fig:comparestudent1}
\end{subfigure}
\hspace{0.85cm}
\begin{subfigure}{0.45\textwidth}
\centering
\includegraphics[scale=0.35]{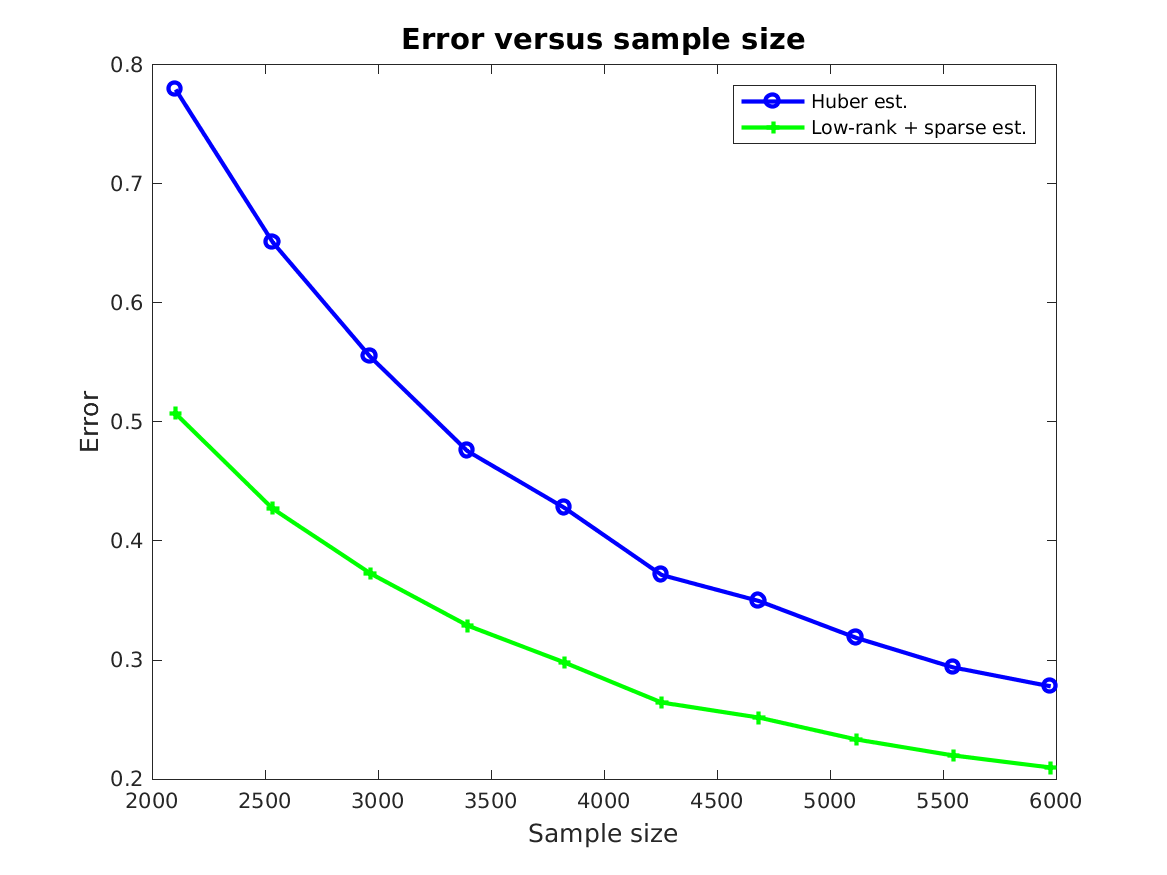}
\caption{}
\label{fig:comparegauss1}
\end{subfigure}
\caption{The left panel shows the Huber estimator \ref{eqn:huber} and the low-rank plus sparse estimator \ref{eqn:kloppest} under the model \ref{eqn:traceregr} with Student t noise with $3$ degrees of freedom. The right panel shows the same estimators under the same model with standard Gaussian noise.}
\end{figure}

In Figure \ref{fig:comparestudent1} the Huber estimator \ref{eqn:huber} is compared with the low-rank plus sparse estimator \ref{eqn:kloppest} under the model \ref{eqn:traceregr} with i.i.d. Student t noise with $3$ degrees of freedom. As expected, these estimators perform comparably well under the trace regression model \ref{eqn:traceregr}. In Figure \ref{fig:comparegauss1} the same estimators are compared under the model \ref{eqn:traceregr} with i.i.d. standard Gaussian noise. Also in this case, we see that both estimators achieve approximately the same error. These observations are not surprising since the theoretical analysis of Section \ref{s:oracle} could be carried over by adapting the (semi-)norms to the different penalization.

We now consider the model proposed in \cite{klopp2014robust} where around $5\%$ of the observed entries are taken to be only one rating. This is the case of malicious users who systematically rate only one particular movie with the same rating. We refer to Section 2.3 of \cite{klopp2014robust} for more details on this setting. In Figure \ref{fig:comparestudent2} we see that the Huber estimator outperforms the low-rank plus sparse estimator with Student t noise with $3$ degrees of freedom. This might be due to the quadratic loss function and to the choice of the tuning parameters. In Figure \ref{fig:comparegauss2} where Gaussian noise is considered we observe that both estimators perform almost equally well.

\begin{figure}[h]
\hspace{-1.5cm}
\centering
\begin{subfigure}{0.45\textwidth}
\centering
\includegraphics[scale=0.35]{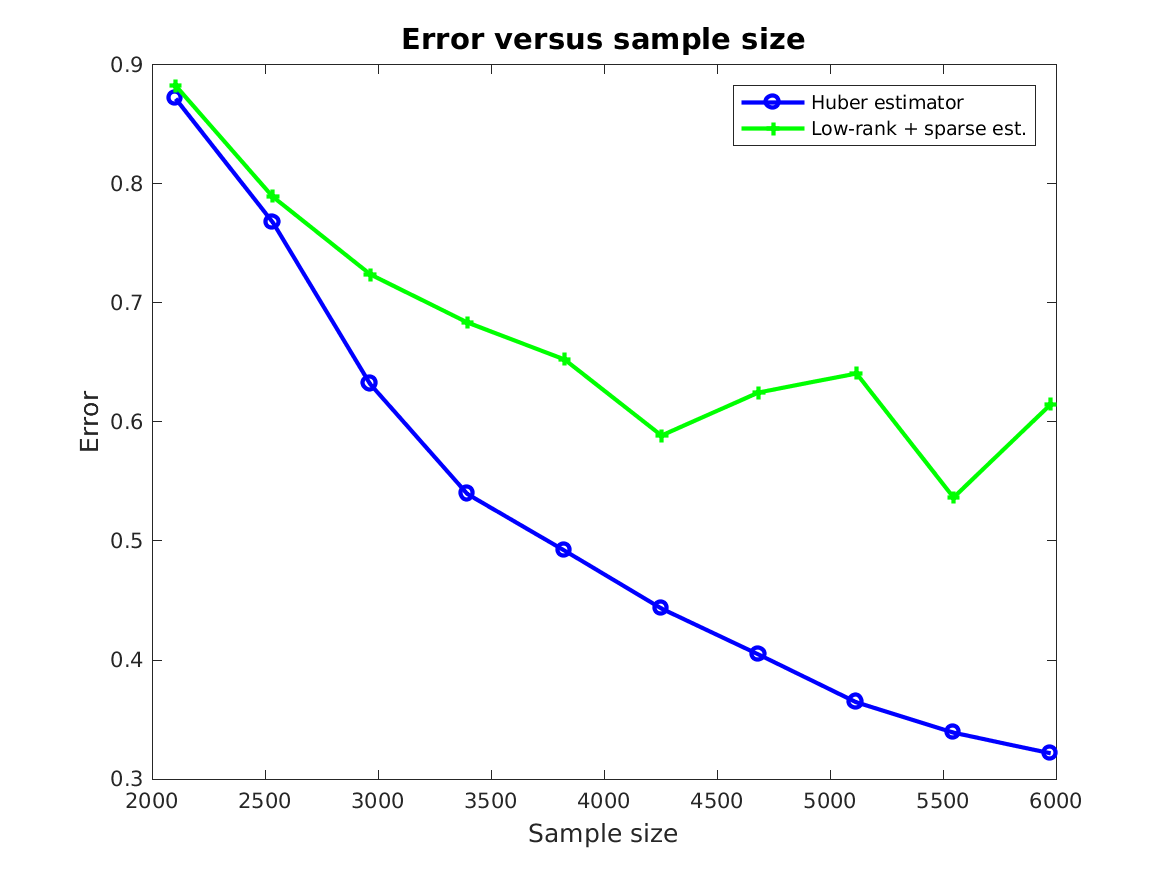}
\caption{}
\label{fig:comparestudent2}
\end{subfigure}
\hspace{0.85cm}
\begin{subfigure}{0.45\textwidth}
\centering
\includegraphics[scale=0.35]{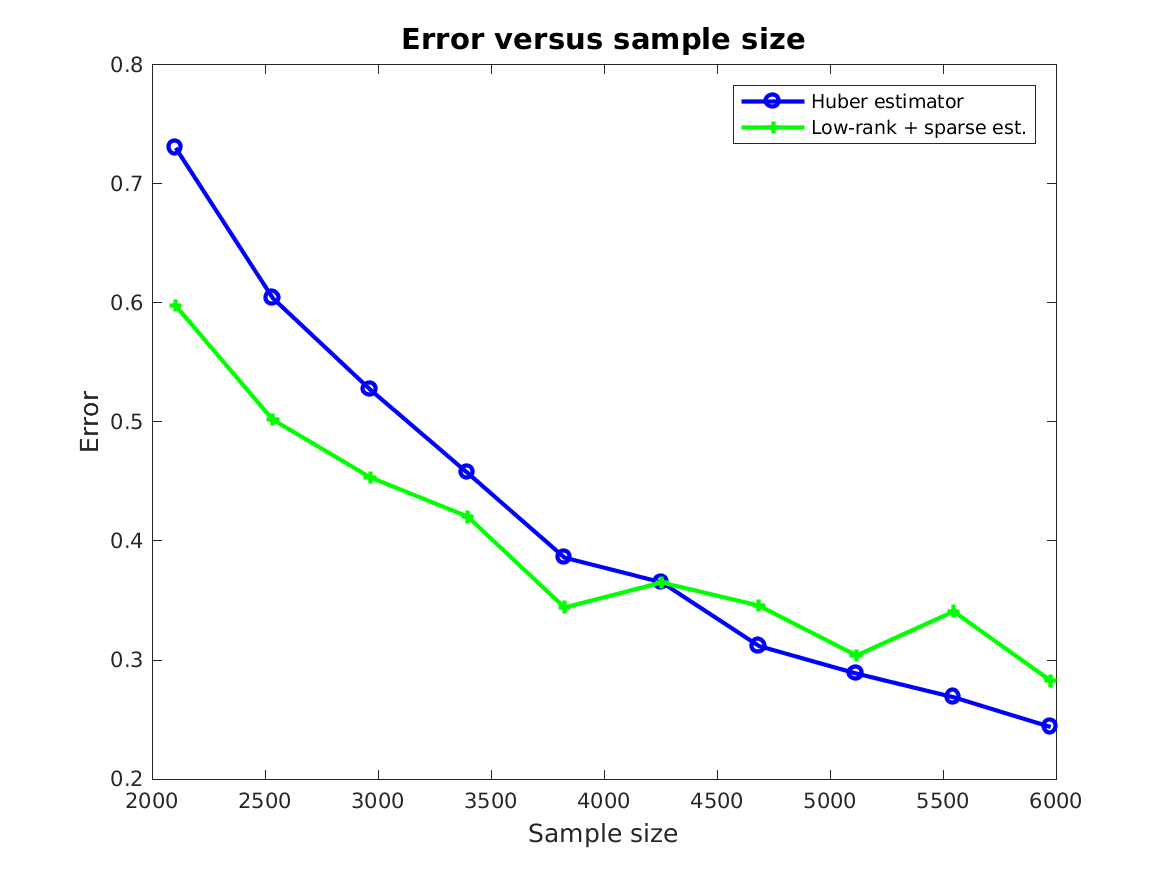}
\caption{}
\label{fig:comparegauss2}
\end{subfigure}
\caption{The left panel shows the Huber estimator \ref{eqn:huber} and the low-rank plus sparse estimator \ref{eqn:kloppest} under the model \ref{eqn:kloppmodel} with Student t noise with $3$ degrees of freedom. The right panel shows the same estimators under the same model as on the left panel but with standard Gaussian noise.}
\end{figure}
\section{Discussion}
In this paper, we have derived sharp and non-sharp oracle inequalities for two robust nuclear norm penalized estimators of the noisy matrix completion problem. The robust estimators were defined using the well-known Huber loss for which the sharp oracle inequality has been derived and the absolute value loss for which we have shown a non-sharp oracle inequality. For both types of oracle inequalities we proved a general deterministic result first and added then the part arising from the empirical process. We have also shown how to apply the oracle inequalities to the case where we only assume weak sparsity, i.e. approximately low-rank matrices. It is worth pointing out that our estimators do not require the distribution on the set of matrices (\ref{eqn:space}) to be known in contrast to e.g. \cite{Koltchinskii2011}. In our case, the distribution on the set of matrices (\ref{eqn:space}) is only needed in the theoretical analysis. The proofs of the oracle inequalities rely on the properties of the nuclear norm, and for the empirical process part on the Concentration, Symmetrization, and Contraction Theorems. A main tool in this context was also the bound on the largest singular value of a matrix with finite Orlicz norm. Our simulations, in the case of the Huber loss, showed a very good agreement with the convergence rates obtained by our theoretical analysis. We saw that the oracle rate is attained up to constants in presence of non-Gaussian noise and that the robust estimation procedure outperforms the quadratic loss function.

It is left to future research to establish a sharp oracle inequality also for the case of a non-differentiable robust loss function. The Contraction inequality used in this paper for the Huber loss requires that also the derivative of the loss is Lipschitz continuous. This is not the case for the absolute value loss. Thanks to the convexity of the loss function it might be possible to derive a sharp result also for this case.

\appendix
\section{Proofs of main results}
\label{appendix:proofmain}
\begin{proof}[Proof of Inequality \ref{lemma:triangleprop}]
Using the triangle property at $B^+$ with $B' = B^+$ we obtain
\begin{align*}
0 = \Vert B^+ \Vert_{\mbox{nuclear}} - \Vert B^+ \Vert_{\mbox{nuclear}} \leq \underbrace{\Omega^+(B^+ - B^+ )}_{=0} - \Omega^-(B^+) \\
\Rightarrow \Omega^-(B^+)=0.
\end{align*}
By the triangle property at $B^+$ with $B' = B= B^+ + B^-$ we have that
\begin{align*}
&\Vert B^+ \Vert_{\mbox{nuclear}}- \Vert B \Vert_{\mbox{nuclear}} \\
&= \Vert B^+ \Vert_{\mbox{nuclear}}- \Vert B^+ + B^- \Vert_{\mbox{nuclear}}	\\
&\leq \Omega^+ (B^- ) - \Omega^-(B^++B^-) \\
&= - \Omega^-(B^+ + B^-).
\end{align*}
By the triangle inequality it follows using $\Omega^-(B^+)=0$ that
\begin{equation*}
\Omega^- (B^+ + B^-) \geq \Omega^-(B^-)  - \Omega^- (B^+) =  \Omega^-(B^-).
\end{equation*}
Therefore, we have
\begin{equation*}
\Vert B^+ \Vert_{\mbox{nuclear}} - \Vert B^+ + B^- \Vert_{\mbox{nuclear}} \leq - \Omega^-(B^-),
\end{equation*}
and by the triangle inequality
\begin{equation*}
\Vert B^+ \Vert_{\mbox{nuclear}} - \Vert B^+ + B^- \Vert_{\mbox{nuclear}} \geq - \Vert B^- \Vert_{\mbox{nuclear}},
\end{equation*}
which gives
\begin{equation*}
\Omega^-(B^-) \leq \Vert B^- \Vert_{\mbox{nuclear}}.
\end{equation*}
For an arbitrary $B$ we have again by the triangle inequality
\begin{equation*}
\Vert B \Vert_{\mbox{nuclear}} - \Vert B' \Vert_{\mbox{nuclear}} \leq \Vert  B^+ \Vert_{\mbox{nuclear}}+  \Vert B^- \Vert_{\mbox{nuclear}} - \Vert B' \Vert_{\mbox{nuclear}} .
\end{equation*}
Applying the triangle property at $B^+$ we find that
\begin{equation*}
\Vert B \Vert_{\mbox{nuclear}} - \Vert B' \Vert_{\mbox{nuclear}} \leq \Omega^+(B^+ - B') - \Omega^-(B') + \Vert B^-\Vert_{\mbox{nuclear}}.
\end{equation*}
Apply now twice the triangle inequality (first inequality) to find that
\begin{align*}
\Vert B \Vert_{\mbox{nuclear}}  - \Vert B' \Vert_{\mbox{nuclear}} &\leq \Omega^+(B - B') + \Omega^+(B^-) - \Omega^-(B - B') \\ 
&\phantom{\leq}+ \Omega^-(B) + \Vert B^- \Vert_{\mbox{nuclear}} \\
&\leq \Omega^+ ( B - B') - \Omega^-(B - B') \\
&\phantom{\leq}+ 2 \Vert B^- \Vert_{\mbox{nuclear}}, 
\end{align*}
where it was used that $\Omega (B^-) = 0$ and that $\Omega^-(B) \leq \Omega^- (B^-) \leq \Vert B^-\Vert_{\mbox{nuclear}}$.
\end{proof}

\begin{proof}[Proof of Lemma \ref{lemma:twopoint}]
Let $B \in \mathcal{B}$. Define for $0<t<1 $
\begin{equation*}
\tilde{B}_t := (1-t)\hat{B} + tB.
\end{equation*}
Since $\mathcal{B}$ is convex we have that $\tilde{B}_t \in \mathcal{B}$ for all $0<t<1 $. Since $\hat{B}$ is the minimizer of the objective function and by the convexity of the objective function we have
\begin{align*}
R_n(\hat{B}) + \lambda \Vert \hat{B} \Vert_{\mbox{nuclear}} &\leq R_n (\tilde{B}_t) + \lambda \Vert \tilde{B}_t \Vert_{\mbox{nuclear}} \\
&\leq R_n (\tilde{B}_t) + (1-t) \lambda \Vert \hat{B} \Vert_{\mbox{nuclear}} + t \lambda \Vert B \Vert_{\mbox{nuclear}}.
\end{align*}
Finally, we can conclude that
\begin{equation*}
\frac{R_n(\hat{B}) - R_n (\tilde{B}_t)}{t} \leq \lambda \Vert B \Vert_{\mbox{nuclear}} - \lambda \Vert \hat{B} \Vert_{\mbox{nuclear}}.
\end{equation*}
Letting $t \rightarrow 0$ the claim follows.
\end{proof}

\begin{proof}[Proof of Theorem \ref{thm:sharp}]
The first order Taylor expansion of $R$ at $\hat{B}$ is given by
\begin{equation}
R(B) = R(\hat{B}) + \trace \left( \dot{R}(\hat{B})^T (B -\hat{B} )  \right) + Rem(\hat{B}, B).
\end{equation}
Then it follows that
\begin{equation}
R(\hat{B} ) - R(B) + Rem(\hat{B}, B) = - \trace \left( \dot{R} (\hat{B})^T (B - \hat{B}) \right).
\end{equation}
\textit{Case 1}

If
\begin{align}
&\trace \left( \dot{R}(\hat{B} )^T (B- \hat{B} ) \right)  \nonumber \\
&\geq \delta \underline{\lambda} \Omega^+(\hat{B} - B) + \delta \underline{\lambda} \Omega^- (\hat{B} - B)  - 2 \lambda \Vert B^- \Vert_{\mbox{nuclear}} -\lambda_*,
\end{align}
then by the two-point-margin condition \ref{ass:2margin} we find that
\begin{equation}
R(B) - R(\hat{B}) \geq \trace \left( \dot{R}(\hat{B})^T (B - \hat{B}) \right) + G(\Vert B-\hat{B} \Vert_F).
\end{equation}
Which implies that
\begin{align*}
R(B) - R(\hat{B}) &\geq \delta \underline{\lambda} \Omega^+(\hat{B} - B) + \delta \underline{\lambda} \Omega^- (\hat{B} - B)  - 2 \lambda \Vert B^- \Vert_{\mbox{nuclear}} \\
&\phantom{\geq}- \lambda_*+ \underbrace{G(\Vert B-\hat{B} \Vert_F )}_{\geq 0} \\
&\geq \delta \underline{\lambda} \Omega^+(\hat{B} - B) + \delta \underline{\lambda} \Omega^- (\hat{B} - B)  - 2 \lambda \Vert B^-\Vert_{\mbox{nuclear}} -  \lambda_*.
\end{align*}
\textit{Case 2}

Assume in the following that
\begin{align}
\label{eqn:case2}
&\trace \left( \dot{R}(\hat{B} )^T (B - \hat{B} ) \right) \nonumber \\
&\leq \delta \underline{\lambda} \Omega^+(\hat{B} - B) + \delta \underline{\lambda} \Omega^- (\hat{B} - B)  - 2 \lambda \Vert B^- \Vert_{\mbox{nuclear}} - \lambda_*,
\end{align}
By the two-point inequality (Lemma \ref{lemma:twopoint}) we have that
\begin{equation}
- \trace \left( \dot{R}_n(\hat{B})^T (B - \hat{B}) \right) \leq \lambda \Vert B \Vert_{\mbox{nuclear}} - \lambda \Vert  \hat{B} \Vert_{\mbox{nuclear}},
\end{equation}
which implies that
\begin{equation}
0 \leq \trace \left( \dot{R}_n(\hat{B})^T (B - \hat{B}) \right) + \lambda \Vert B \Vert_{\mbox{nuclear}} - \lambda \Vert \hat{B} \Vert_{\mbox{nuclear}}.
\end{equation}
Hence,
\begin{align*}
&-\trace \left( \dot{R}(\hat{B})^T (B - \hat{B}) \right)+ \delta \underline{\lambda} \Omega^+(\hat{B} - B) + \delta \underline{\lambda} \Omega^- (\hat{B} - B)\\
&\leq \trace \left( (\dot{R}_n (\hat{B}) - \dot{R} (\hat{B}))^T (B - \hat{B}) \right)  +\delta \underline{\lambda} \Omega^+(\hat{B} - B) + \delta \underline{\lambda} \Omega^- (\hat{B} - B) \\
&\phantom{\leq} + \lambda \Vert B \Vert_{\mbox{nuclear}} - \lambda \Vert \hat{B} \Vert_{\mbox{nuclear}} \\
&\leq  \lambda_\varepsilon \underline{\Omega}(\hat{B}-B) + \lambda_*+ \delta \underline{\lambda} \Omega^+(\hat{B} - B) + \delta \underline{\lambda} \Omega^- (\hat{B} - B) \\
&\phantom{\leq}+ \lambda \Vert B \Vert_{\mbox{nuclear}} - \lambda \Vert \hat{B} \Vert_{\mbox{nuclear}} \\
&\leq \lambda_\varepsilon \Omega^+(\hat{B}-B) +\lambda_\varepsilon  \Omega^- (\hat{B} - B) + \lambda_*+ \delta \underline{\lambda} \Omega^+(\hat{B} - B) \\
&\phantom{\leq} + \delta \underline{\lambda} \Omega^- (\hat{B} - B) + \lambda \Omega^+(\hat{B} - B) - \lambda \Omega^-(\hat{B} - B) + 2 \lambda \Vert B^- \Vert_{\mbox{nuclear}} \\
&= \overline{\lambda} \Omega^+ (\hat{B} - B) - (1 -\delta)\underline{\lambda} \Omega^-(\hat{B} - B) + 2\lambda \Vert B^- \Vert_{\mbox{nuclear}} +\lambda_*.
\end{align*}
Therefore, by Equation \ref{eqn:case2}
\begin{equation*}
\Omega^-(\hat{B} - B) \leq \frac{\overline{\lambda}}{(1- \delta)\underline{\lambda}} \Omega^+(\hat{B} - B).
\end{equation*}
We then have by the convex conjugate inequality
\begin{align*}
 \Omega^+(\hat{B} -B) &\leq \Vert \hat{B} -B \Vert_F 3 \sqrt{s} \\
&\leq H\left(3 \sqrt{s} \right) + G(\Vert \hat{B} - B \Vert_F ).
\end{align*}
Which implies that
\begin{align*}
&-\trace \left( \dot{R} (\hat{B})^T (B - \hat{B}) \right) +  \underline{\lambda} \Omega^-(\hat{B} - B) + \delta \underline{\lambda} \Omega^+(\hat{B} - B)\\
&= R(\hat{B})-R(B) + Rem(\hat{B}, B) + \underline{\lambda}  \Omega^- (\hat{B} - B) + \delta \underline{\lambda} \Omega^+ (\hat{B} - B)  \\
&\leq H \left( \overline{\lambda} 3 \sqrt{s} \right) + G(\Vert \hat{B} - B \Vert_F) + 2\lambda \Vert B^- \Vert_{\mbox{nuclear}} + \lambda_*\\
&\leq H\left(\overline{\lambda}3 \sqrt{s} \right) + Rem(\hat{B}, B) + 2\lambda \Vert B^- \Vert_{\mbox{nuclear}} + \lambda_*.
\end{align*}
\end{proof}

\begin{proof}[Proof of Theorem \ref{thm:nonsharp}]
	We start the proof with the following inequality using the fact that $\hat{B}$ is the minimizer of the objective function.
	\begin{equation}
	R_n(\hat{B}) + \lambda \Vert \hat{B} \Vert_{\mbox{nuclear}} \leq R_n (B) + \lambda \left\Vert B \right\Vert_{\mbox{nuclear}}
	\end{equation}
	Then, by adding and subtracting $R(\hat{B})$ on the left hand side and $R(B)$ on the right hand side we obtain
	\begin{align*}
	R(\hat{B}) - R(B) &\leq - \left[ (R_n(\hat{B}) - R(\hat{B}) - ( R_n (B)- R(B))\right]  \\
	&\phantom{\leq}+ \lambda \Vert B \Vert_{\mbox{nuclear}} - \lambda \Vert \hat{B} \Vert_{\mbox{nuclear}}.
	\end{align*}
	Applying Assumption \ref{eqn:assbound}, the definition of $\underline{\Omega}$, and Lemma \ref{lemma:triangleprop} we obtain
	\begin{align*}
	&R(\hat{B}) - R(B) \leq \lambda_\varepsilon \underline{\Omega} (\hat{B} - B) + \lambda_* + \lambda \Vert B \Vert_{\mbox{nuclear}} - \lambda \Vert \hat{B} \Vert_{\mbox{nuclear}} \\
	&\leq \lambda_\varepsilon \Omega^+(\hat{B} - B) + \lambda_\varepsilon \Omega^-(\hat{B} - B) + \lambda_*\\
	&\phantom{\leq} + \lambda \Omega^+ (\hat{B} - B) - \lambda \Omega^-(\hat{B} - B) + 2 \lambda \Vert B^- \Vert_{\mbox{nuclear}} \\
	&= (\lambda_\varepsilon + \lambda) \Omega^+ (\hat{B} - B)  - (\lambda - \lambda_\varepsilon) \Omega^-(\hat{B} - B) + \lambda_* + 2 \lambda \Vert B^- \Vert_{\mbox{nuclear}}.
	\end{align*}
	Since later on we apply Assumption \ref{ass:1margin} we subtract on both sides of the above inequality $R(B^0)$.
	\begin{align}
	\label{eqn:basic1}
	&R(\hat{B}) - R(B^0) + \underline{\lambda} \Omega^- (\hat{B} - B) \nonumber \\
	&\leq R(B) - R(B^0) +\overline{\lambda} \Omega^+ (\hat{B} - B) + \lambda_*+ 2 \lambda \Vert B^- \Vert_{\mbox{nuclear}}.
	\end{align}
	It is then useful to make the following case distinction that allows us to obtain an upper bound for the estimation error.
	\\
	\textit{Case 1}
	\\
	If $\overline{\lambda} \Omega^+ (\hat{B} -B)\leq \frac{(1-\delta)}{\delta} \left( \lambda_*+ R(B) - R(B^0) + 2 \lambda \Vert B^- \Vert_{\mbox{nuclear}} \right)$, then
	\begin{align*}
	 \delta \overline{\lambda} \Omega^+ (\hat{B}- B) &\leq (1-\delta) \left(\lambda_*+ R(B) - R(B^0) + 2 \lambda \Vert B^- \Vert_{\mbox{nuclear}}\right) \\
	 &\leq \lambda_*+ R(B) - R(B^0) + 2 \lambda \Vert B^- \Vert_{\mbox{nuclear}} .
	\end{align*}
	By multiplying Equaiton \ref{eqn:basic1} on both sides with $\delta$ we arrive at
	\begin{align*}
	\delta \underline{\lambda} \Omega^- (\hat{B} - B) \leq \lambda_* + R(B) - R(B^0) + 2\lambda \Vert B^- \Vert_{\mbox{nuclear}}.
	\end{align*}
	Therefore,
	\begin{align}
	&\delta (\overline{\lambda} \Omega^+ (\hat{B}- B) + \underline{\lambda} \Omega^- (\hat{B}- B)) \nonumber \\ 
	&\leq 2\lambda_*+ 2 (R(B) - R(B^0)) + 4 \lambda \Vert B^- \Vert_{\mbox{nuclear}}.
	\end{align}
	And since 
	\begin{equation*}
	\underline{\lambda} < \overline{\lambda},
	\end{equation*}
	we conclude that
	\begin{equation*}
	\delta \underline{\lambda} (\Omega^+ + \Omega^- ) (\hat{B} - B) \leq 2 \lambda_* + 2(R(B)-R(B^0)) + 4 \lambda \Vert B^- \Vert_{\mbox{nuclear}}.
	\end{equation*}
	\textit{Case 2}
	\\
	If $\overline{\lambda} \Omega^+(\hat{B} - B) \geq \frac{(1-\delta)}{\delta} \left(\lambda_*+ R(B) - R(B^0) + 2 \lambda \Vert B^- \Vert_{\mbox{nuclear}} \right) $, then
	\begin{align*}
	R(\hat{B}) - R(B^0) + \underline{\lambda} \Omega^- (\hat{B} - B) &\leq \overline{\lambda} \Omega^+ (\hat{B} - B) + \overline{\lambda}\Omega^+ (\hat{B} - B) \frac{\delta}{(1-\delta)}.
	\end{align*}
	This implies
	\begin{align*}
	(1-\delta)  \left[ R(\hat{B}) - R(B^0) \right] + (1-\delta) \underline{\lambda} \Omega^- (\hat{B} - B) 	\leq \overline{\lambda} \Omega^+ (\hat{B} - B).
	\end{align*}
	And finally we conclude that 
	\begin{equation*}
	\Omega^- (\hat{B} - B) \leq \frac{\overline{\lambda}}{(1-\delta) \underline{\lambda}} \Omega^+(\hat{B} - B).
	\end{equation*}
	We then obtain using the definition of $\Omega^+$ in Lemma \ref{def:triangle}
	\begin{align*}
	\Omega^+(\hat{B}-B) &\leq \Vert \hat{B} - B \Vert_F 3 \sqrt{s} \\
	&\leq (\Vert \hat{B} - B^0 \Vert_F + \Vert B - B^0 \Vert_F ) 3 \sqrt{s} \\
	&\leq G(\Vert \hat{B} - B^0  \Vert_F) + G(\Vert B - B^0 \Vert_F)  + 2 H(3 \sqrt{s} ). 
	\end{align*}
	Invoking the convex conjugate inequality and Assumption \ref{ass:1margin} we get
	\begin{align*}
	&\delta \overline{\lambda} \Omega^+ (\hat{B} - B) + \delta \underline{\lambda} \Omega^- (\hat{ B} - B) \\
	&\leq 2H(\overline{\lambda} (1+\delta) 3 \sqrt{s}) + R(B) - R(B^0) + (R(B) - R(B^0)) \\
	&\phantom{\leq} + \lambda_* + 2 \lambda \Vert B^- \Vert_{\mbox{nuclear}} \\
	& \leq 2H(\overline{\lambda} (1+\delta)3\sqrt{s}) +2 (R(B) - R(B^0)) \\
	&\phantom{\leq}+ \lambda_*+ 2 \lambda \Vert B^- \Vert_{\mbox{nuclear}}.
	\end{align*}
	Combining the two cases we have for the estimation error
	\begin{align*}
	&\delta \underline{\lambda} (\Omega^+ + \Omega^- ) (\hat{B} - B) \\
	&\leq 2H(\overline{\lambda} (1+\delta)3\sqrt{s}) + 2\lambda_* \\
	&\phantom{\leq}+2(R(B) - R(B^0)) + 4 \lambda \Vert B^- \Vert_{\mbox{nuclear}}.
	\end{align*}
	and for the second claim we conclude that
\begin{align*}
	R(\hat{B}) - R(B) &\leq \overline{\lambda} \Omega^+ (\hat{B} - B) + \lambda_* + 2 \lambda \Vert B^- \Vert_{\mbox{nuclear}} \\
	&\leq \frac{1}{\delta} \left[ 2H(\overline{\lambda} (1+\delta) 3\sqrt{s}) + \lambda_* + 2 (R(B)-R(B^0)) \right. \\
	&\left. \phantom{\leq}+ 2 \lambda \Vert B^- \Vert_{\mbox{nuclear}} \right] + \lambda_* + 2 \lambda \Vert B^- \Vert_{\mbox{nuclear}}.
\end{align*}
\end{proof}
\begin{proof}[Proof of Lemma \ref{lemma:2margin}]
		The theoretical risk function arising from the Huber loss is given by
		\begin{equation}
		R(B) = \frac{1}{n} \sum_{i=1}^n \mathbb{E}_{X_i} \left[ \mathbb{E} \left[ \rho_H \left( Y_i - \trace(X_i B) \right) \vert X_i \right] \right].
		\end{equation}
		Suppose that $X_i$ has its only $1$ at entry $(k,j)$. Then $XB = (B)_{jk}$. Define	
		\begin{align*}
		r(x,B) &:= \mathbb{E} \left[ \rho_H \left( Y_i - \trace \left(X_i B \right) \right) \vert X_i = x \right] \\
		&\phantom{:}= \mathbb{E} \left[ \rho_H \left( Y_i - B_{jk} \right) \right]
		\end{align*}
		We notice that $\dot{r}(x,B) = \frac{d r(x,B)}{dB_{jk}}= \mathbb{E} \left[ \frac{d\rho_H(Y_i - B_{jk})}{dB_{jk}}  \right]$. The derivative with respect to $B_{jk}$ of $\rho_H(Y_i - B_{jk})$ is given by
		\begin{equation*}
		\dot{\rho}_H(Y_i - B_{jk}) := \left\lbrace \begin{array}{ll}
		-2(Y_i - B_{jk}), & \mbox{if} \ \left\vert Y_i - B_{jk}\right\vert \leq \kappa \\
		-2\kappa, & \mbox{if} \  Y_i - B_{jk}  > \kappa \\
		2\kappa, & \mbox{if} \ Y_i - B_{jk} < - \kappa.
		\end{array} \right.
		\end{equation*}
		Then,
		\begin{align*}
		\dot{r}(x,B) &= - 2 \int_{B_{jk} - \kappa}^{B_{jk} + \kappa} (y - B_{jk}) dF(y) - 2\kappa \int_{B_{jk} + \kappa}^\infty dF(y)  + 2 \kappa \int_{-\infty}^{B_{jk} - \kappa } dF(y)\\
		&=- 2 \int_{B_{jk}-\kappa}^{B_{jk} + \kappa} y dF(y) + 2 B_{jk} \int_{B_{jk}-\kappa}^{B_{jk}+\kappa} dF(y) - 2\kappa \left[1 - F(\kappa+ B_{jk})\right] \\
		&\phantom{=}+ 2 \kappa F(B_{jk}-\kappa) \\
		&= -2 (B_{jk} + \kappa ) F(B_{jk} + \kappa) + 2(B_{jk} -\kappa) F(B_{jk}-\kappa) \\
		&\phantom{=} + 2 \int_{B_{jk} - \kappa}^{B_{jk}+\kappa} F(y) dy + 2B_{jk} \left[F(B_{jk} + \kappa) - F(B_{jk} - \kappa )\right] - 2 \kappa \\
		&\phantom{=} + 2 \kappa F(\kappa + B_{jk}) + 2 \kappa F(B_{jk} - \kappa).\\
		&=2 \int_{B_{jk} - \kappa}^{B_{jk} + \kappa} F(y) dy - 2\kappa.
		\end{align*}
		The second derivative of $r(x,B)$ with respect to $B_{jk}$ is then given by
		\begin{equation*}
		\ddot{r} (x,B) = 2 [F(B_{jk} + \kappa) - F(B_{jk} - \kappa)].
		\end{equation*}
		Therefore, the Taylor expansion around $B'$ is given by
		\begin{align*}
		r(x,B) = r(x,B') + \dot{r}(x,B') (B_{jk} - B_{jk}') + \frac{\ddot{r}(x,\tilde{B})}{2} (B_{jk} - B_{jk}')^2,
		\end{align*}
		where $\tilde{B} \in \mathcal{B}$ is an intermediate point.
		
		We can see that Assumption \ref{ass:2margin} holds with $G(u) = u^2/(2C_1^2 pq)$. 
\end{proof}

\begin{proof}[Proof of Lemma \ref{lemma:1margin}]
	For the (theoretical) risk function $R$ arising from the absolute value loss we have
	\begin{align}
	R(B) &= \mathbb{E} \left[ R_n (B) \right] \\
	&= \frac{1}{n} \sum_{i=1}^n \mathbb{E} \left[ \left\vert Y_i - \trace \left(X_i B \right) \right\vert \right]. \\
	\end{align}
	Using the tower property of the conditional expectation we obtain
	\begin{equation}
	R(B) = \frac{1}{n} \sum_{i=1}^n \mathbb{E}_{X_i} \left[ \mathbb{E} \left[ \left\vert Y_i - \trace(X_i B) \right\vert \vert X_i \right] \right].
	\end{equation}
	Suppose that $X_i$ has its only $1$ at entry $(k,j)$. Then $XB = (B)_{jk}$. Define
	\begin{align*}
	r(x,B) &:= \mathbb{E} \left[ \left\vert Y_i - \trace \left(X_i B \right) \right\vert \vert X_i = x \right] \\
	&\phantom{:}= \mathbb{E} \left[ \left\vert Y_i -B_{jk} \right\vert \right] \\
	&\phantom{:}= \int_{y \geq B_{jk}} \left( y - B_{jk} \right) dF(y) + \int_{y < B_{jk}} \left( B_{jk} - y \right) dF(y) \\
	&\phantom{:}= \int_{y \geq B_{jk}} \left( y - B_{jk} \right) dF(y) + \int_{-\infty}^\infty \left( B_{jk} - y \right) dF(y)\\
	&\phantom{:=} - \int_{B_{jk}}^\infty \left( B_{jk} - y \right) dF(y) \\
	&\phantom{:}= 2 \int_{B_{jk}}^\infty \left( y - B_{jk} \right) dF(y) + \int_{-\infty}^\infty	\left( B_{jk} - y \right) dF(y) \\
	&\phantom{:}= 2 \int_{B_{jk}}^\infty \left( y - B_{jk} \right) dF(y)  + B_{jk} \underbrace{\int_{-\infty}^\infty dF(y)}_{= 1} - \int_{-\infty}^\infty y dF(y) \\
	&\phantom{:}= 2\int_{B_{jk}}^\infty \left( 1 - F(y) \right) dy + B_{jk} - \int_{-\infty }^\infty y dF(y).
	\end{align*}
	
	The Taylor expansion of $r(x,B)$ around $B^0$, assuming that $B^0$ minimizes $r$, is given by
	\begin{align*}
	r(x,B) &= r(x,B^0) + \dot{r} (x, B^0) \left(B_{jk} - B^0_{jk} \right) + \frac{\ddot{r}(x,\tilde{B})}{2} \left( B_{jk} - B_{jk}^0 \right)^2 \\
	&= r(x,B^0) + f(\tilde{B}_{jk} ) \left(B_{jk} - B_{jk}^0 \right)^2,
	\end{align*}
	where $\tilde{B} \in \mathcal{B}$ is an intermediate point.
	\begin{equation*}
	r(x,B) - r(x,B^0) \geq  \frac{1}{C_2^2} \left(B_{jk} - B_{jk}^0 \right)^2
	\end{equation*}
	which means that the one point margin Condition \ref{ass:1margin} is satisfied with $G(u) = u^2/(2C_2^2pq)$. 
\end{proof}



\section{Supplemental Material}
\label{s:supplement}
This supplemental material contains an application to real data sets, the proofs of the lemmas in Section \ref{s:nuclear} of the main text, and a section on the bound of the empirical process part of the estimation problem.
\subsection{Example with real data}
In Section 5 we have shown several synthetic data examples. The convex optimization problems there were solved using the semidefinite programming (SDP) toolbox \cite{cvx}. These algorithms work very well for comparably low-dimensional optimization problems. When real datasets are considered, due to the much larger problem sizes different algorithms are needed. To solve the optimization problem with real data a proximal gradient algorithm is used. The algorithm is given in pseudocode.

We define $F_1$ to be the empirical risk for the Huber loss
$$F_1(B) = \frac{1}{n} \sum_{i = 1}^n \rho_H(Y_i - \trace(X_i B)).$$

We define $F_2$ to be the empirical risk for the quadratic loss
$$F_2(B) = \frac{1}{n} \sum_{i = 1}^n (Y_i - \trace(X_i B))^2.$$

The gradient of $F_1$ is given by
$$\triangledown F_1(B) = -\frac{1}{n} \sum_{i = 1}^n \dot{\rho}_H(Y_i - \trace(X_i B))X_i^T, $$
where $\dot{\rho}_H(Y_i - \trace(X_i B))$ is given in the proof of Lemma 2.1 in the main paper.

The gradient of $F_2$ is given by

$$\triangledown F_2(B) = -\frac{2}{n} \sum_{i = 1}^n (Y_i - \trace(X_i B)) X_i^T.$$

The proximity operator proxnuc for $W \in \mathbb{R}^{p \times q}$ and $\gamma > 0$ is defined as
\begin{equation*}
\text{proxnuc}_\gamma (W)	 = \underset{B \in \mathbb{R}^{p \times q} }{\arg \min} \ \gamma \Vert B \Vert_{\text{nuclear}} + \frac{1}{2} \Vert B - W \Vert_F^2.
\end{equation*}

For the nuclear norm the proximity operator has a closed form: let $W = U \text{diag}(\sigma_1, \dots, \sigma_{\min(p,q)}) V'$ be the singular value decomposition of $W$, then
\begin{equation*}
\text{proxnuc}_\gamma (W) = U \mathcal{S}_\gamma \left((\sigma_1, \dots, \sigma_{\min(p,q)}) \right) V',
\end{equation*}
where for $x \in \mathbb{R}$ the operator $\mathcal{S}_\gamma(x)$ applied elementwise is given by
\begin{equation*}
\mathcal{S}_\gamma(x) = \left\lbrace \begin{array}{cc}
x - \gamma, &\text{if} \ x > \gamma \\
0, &\text{if} \ x = \gamma \\
x + \gamma, &\text{if} \ x < \gamma.
\end{array} \right.
\end{equation*}

It is known that the solution of the optimization problem $\hat{B}_H$ satisfies the following fixed point equation

\begin{equation*}
\hat{B}_H = \text{proxnuc}_\gamma \left(\hat{B}_H - \gamma \triangledown F \right).
\end{equation*}
The same holds also for the quadratic loss function. To compute the proximal operator in Algorithm \ref{Algorithm 1} the function \linebreak \texttt{prox\underline{\phantom{a}}nuclearnorm} from the Matlab toolbox \cite{unlocbox} was used. The algorithm is a Nesterov-type Accelerated Proximal Gradient algorithm. We refer to Section 4.3 in \cite{parikh2014proximal} and references therein. In particular, the present form of Algorithm \ref{Algorithm 1} goes back to \cite{beck2009gradient}. The Huber constant is chosen to be $\kappa = 2$ in all the examples. Algorithm \ref{Algorithm 1} is applied to the Huber loss ($i = 1$) and to the quadratic loss ($i = 2$). The tuning parameter for the quadratic loss case is smaller than for the Huber loss.

\begin{algorithm} 

\caption{HuberQuadProx} 
\label{Algorithm 1} 

\begin{algorithmic}
\STATE Start with an initial value $B = B_{\text{obs}}$ containing the entries belonging to the training set. Define the dummy variable $v= B$. Choose $L = 0.1$, $\beta = 1.2$, $\lambda_1 = 20 \sqrt{\log(p+q)/(nq)}$, and $\lambda_2 = \sqrt{\log(p+q)/(nq)}$.
\WHILE{Stopping criterion is not satisfied OR $t \leq \text{maxiter}$ } 
\STATE $\text{obj} \leftarrow F_i(B) + \lambda_i \Vert B \Vert_{\text{nuclear}} $
\STATE $B_{\text{prev}} \leftarrow B$
\STATE $\delta = 1$
\WHILE{$\delta > 0.001$}
\STATE $B \leftarrow v - 1/L \triangledown F_i (v)$
\STATE $ B \leftarrow \text{proxnuc}_{\lambda/L} (B)$
\STATE $\delta \leftarrow F_i(B) + \lambda \Vert B \Vert_{\text{nuclear}} - \text{obj} - \trace(\triangledown F_i(B_{\text{prev}})^T (B - B_{\text{prev}}))- L/2\Vert B - B_{\text{prev}} \Vert_F^2$
\vspace{-4mm}
\STATE $L \leftarrow \beta L$
\ENDWHILE
\STATE $L \leftarrow L/\beta$
\STATE 
$v \leftarrow B + t/(t+3) (B - B_{\text{prev}})$
\ENDWHILE

\end{algorithmic}
\end{algorithm} 
This method was applied to the MovieLens data set which consists of $p = 943$ users, $q = 1682$ movies, and $100'000$ observed ratings, we call the set containing the indices of the observed ratings $\Gamma$. The minimal and maximal ratings are $1$ and $5$ respectively. Every user has rated at least $20$ movies. The data set is available at \cite{movielens}. Training and testing sets of different sizes are drawn randomly without replacement from the set of observed ratings. The testing error is computed as follows: with $100'000 = n_{\text{test}} + n_{\text{train}}$ and using the disjoint union $\Gamma = \Gamma_{\text{test}} \cup \Gamma_{\text{train}}$ (where $\vert \Gamma_{\text{test}} \vert = n_{\text{test}}$ and $\vert \Gamma_{\text{train}} \vert = n_{\text{train}}$) we have
\begin{equation*}
\text{test error} = \frac{1}{n_{\text{test}}} \sum_{(j,k) \in \Gamma_{\text{test}}} (B^0_{j,k} - \hat{B}_{j,k})^2.
\end{equation*}
\newpage

The results are reported in Table \ref{tbl:resmov}:
\begin{center}
\begin{table}[H]
\begin{tabular}{|c|c|c|c|c|c|}
\hline
\multicolumn{2}{|c|}{}&\multicolumn{2}{|c|}{Huber loss} & \multicolumn{2}{|c|}{Quadratic loss}\\ \hline
training set size & test set size & test error & Iterations & test error & Iterations\\ \hline
$25'000 $& $75'000$ & $1.43$ & $6'000$& $1.32 $ & $6'000$ \\
$50'000 $& $50'000$ & $1.14$ & $6'000$ & $1.06 $ & $6'000$ \\ 
$75'000 $& $25'000$ & $1.01$ & $6'000$ & $0.96 $  & $6'000$ \\ \hline
\end{tabular}
\caption{ }
\label{tbl:resmov}
\end{table}
\end{center}
It can be observed that the test error of the quadratic loss estimator is slightly smaller than the test error for the Huber loss estimator. This might be due to the fact that there are no heavily corrupted entries in the data. On the other hand, this indicates that the Huber estimator is able to ``adapt'' also to the usual setting without heavy corruptions.
 
We have applied Algorithm \ref{Algorithm 1} also to the MovieLens data set with $1'000'209$ observed ratings from $6'040$ users on $3'952$ movies. The minimal and maximal ratings are $1$ and $5$ respectively. Every user has rated at least $20$ movies. The data set is available at \cite{movielens1m}.

The results are reported in Table \ref{tbl:resmov1m}:
\begin{center}
\begin{table}[H]
\begin{tabular}{|c|c|c|c|c|c|}
\hline
\multicolumn{2}{|c|}{ }&\multicolumn{2}{|c|}{Huber loss} & \multicolumn{2}{|c|}{Quadratic loss}\\ \hline
training set size & test set size & test error & Iterations& test error & Iterations\\ \hline
$250'000 $& $750'209$ & $1.05$ & $10'000$ & $1.04$ & $10'000$ \\
$500'000 $& $500'209$ & $0.92 $ & $10'000$ & $0.90$ & $10'000$ \\ 
$750'000 $& $250'209$ & $0.85$ & $10'000$& $0.84$ &$10'000$\\ \hline
\end{tabular}
\caption{ }
\label{tbl:resmov1m}
\end{table}
\end{center}
\subsection{Proofs of Lemmas in Section 2}
The following lemma shows ``equivalence'' of the nuclear norm and Frobenius norm. This fact is useful since it is more common and meaningful to measure the estimation error in the Frobenius norm. The rest of the lemma contains technicalities that are used in the sequel to derive the triangle property.
\begin{lemma}
\label{lemma:normequiv}
Let $A \in \mathbb{R}^{p \times q}$. Then 
\begin{equation*}
\left\Vert A \right\Vert_F \leq \left\Vert A \right\Vert_{\mbox{\emph{nuclear}}} \leq \sqrt{\rank{A}} \left\Vert A \right\Vert_F.
\end{equation*}
Let $P \in \mathbb{R}^{p \times s}$ with $P^TP = I$ and $s \leq p$. Then 
\begin{equation*}
\left\Vert PP^T A \right\Vert_F \leq \sqrt{s} \Lambda_{\mbox{\emph{max}}} \left(A \right)
\end{equation*}
and
\begin{equation*}
\left\Vert PP^TA \right\Vert_F \leq \left\Vert A \right\Vert_F, \quad \Vert A QQ^T \Vert_F \leq \Vert A \Vert_F, \quad \Vert A PP^T \Vert_F \leq \Vert A \Vert_F,
\end{equation*}
where $\left\Vert A \right\Vert_F$ is the Frobenius norm of $A$ and $\Lambda_{\mbox{\emph{max}}}(A)$ its largest singular value.
\end{lemma}

\begin{proof}[Proof of Lemma \ref{lemma:normequiv}]
Consider the singular value decomposition (SVD) of $A$ with $s:=\mbox{rank}(A)$
\begin{equation}
A= P_A \Lambda_A Q_A^T
\end{equation}
with $P_A^T P_A = \mathbbm{1}_{s \times s}$, $Q_A^T Q_A = \mathbbm{1}_{s \times s}$ and $\Lambda_A = \diag \left(\Lambda_{A,1}, \dots, \Lambda_{A,s} \right)$. Then the nuclear norm of $A$ can be written as
\begin{equation}
\left\Vert A \right\Vert_{\mbox{nuclear}} = \sum_{k=1}^s  \Lambda_{A,k}.
\end{equation}

The first claim follows from H\"older's inequality applied in two different manners for the lower and upper bounds.

For the second claim consider the $p-$dimensional  $j-$th unit vector $e_j$. Define
\begin{equation*}
u := PP^T e_j.
\end{equation*}
Then
\begin{align*}
u^T A A^T u = \frac{\left\Vert A^T u \right\Vert_2^2}{\left\Vert u \right\Vert_2^2} \left\Vert u \right\Vert_2^2 &\leq \left( \underset{\left\Vert u \right\Vert_2 =1}{\max} \ \left\Vert A^T u \right\Vert_2 \right)^2 \left\Vert u \right\Vert_2^2 \\
&= \Lambda_{\max}^2 \left( A \right) \left\Vert PP^T e_j \right\Vert_2^2.
\end{align*}
Then, by the invariance of the trace under cyclic permutations we obtain the claimed result.
\end{proof}

In order to show that the triangle property holds we also need the dual norm of the nuclear norm so that to apply the dual norm inequality (see Lemma \ref{lemma:subgrad}). The subdifferential of the nuclear norm is then used to deduce the triangle property.

The following lemma gives the dual norm and the subdifferential of the nuclear norm. It cites results of \cite{lange2013optimization} and \cite{watson1992characterization}.
\begin{lemma}
\label{lemma:subgrad}
Let $A \in \mathbb{R}^{p \times q}$.The dual norm of the nuclear norm \linebreak $\left\Vert A \right\Vert_{\mbox{\emph{nuclear}}} $ is given by
\begin{equation*}
\Omega_* (A) = \Lambda_{\max} (A),
\end{equation*}
where $\Lambda_{\max}(A)$ is the largest singular value of $A$ .
Moreover, the subgradient of the nuclear norm is given by
\begin{align*}
\partial \left\Vert B \right\Vert_{\mbox{\emph{nuclear}}} = \left\lbrace Z = PQ^T + \left( \mathbbm{1}_{p \times p } - PP^T \right) W \left( \mathbbm{1}_{q \times q} - QQ^T \right): \right. \\
\left. \Lambda_{\max} \left( W \right) = 1 \right\rbrace.
\end{align*}
\end{lemma}
\begin{proof}
The derivation of the dual norm of the nuclear norm can be found in Example 14.3.6. from \cite{lange2013optimization}. A justification for the second claim can be found in \cite{watson1992characterization}.
\end{proof}

\begin{proof}[Proof of Lemma 2.1 in the main text]
Let $Z \in \partial	\left\Vert B^+ \right\Vert_{\mbox{nuclear}}$, i.e.
\begin{equation*}
Z = P^+Q^{+^T} + \left( \mathbbm{1}_{p \times p } - P^+P^{+^T} \right) W \left(\mathbbm{1}_{q \times q} - Q^+Q^{+^T} \right),
\end{equation*}
where $W \in \mathbb{R}^{p \times q}$ is such that $\Lambda_{\max} \left( W \right) = 1$. Therefore, it is possible to write
\begin{align*}
Z = Z_1 + Z_2, \quad \mbox{where} \ Z_1 &:= P^+Q^{+^T} \ \\
\mbox{and} \ Z_2 &:= \left( \mathbbm{1}_{p \times p } - P^+P^{+^T} \right) W \left(\mathbbm{1}_{q \times q} - Q^+Q^{+^T} \right).
\end{align*}
Recall the definition of the subdifferential of the nuclear norm
\begin{align*}
\partial \left\Vert B^+ \right\Vert_{\mbox{nuclear}} = \left\lbrace Z \vert \left\Vert B' \right\Vert_{\mbox{nuclear}} - \left\Vert B^+ \right\Vert_{\mbox{nuclear}} \geq \trace\left( Z^T \left( B'-B^+ \right) \right), \right. \\
\left. \forall B' \in \mathbb{R}^{p \times q} \right\rbrace
\end{align*}
For $ Z \in \partial \left\Vert B^+ \right\Vert_{\mbox{nuclear}}$ we have 
\begin{align*}
\left\Vert B^+ \right\Vert_{\mbox{nuclear}} - \left\Vert B' \right\Vert_{\mbox{nuclear}} &\leq \trace \left( Z^T \left(B^+- B' \right) \right)\\
&= \trace \left( \left( Z_1 + Z_2 \right)^T \left(B^+-B' \right) \right).
\end{align*}
To prove the first assertion we need to bound the right hand side of the above inequality. For simplicity consider first $\trace \left( Z_1^T B' \right)$. Using the invariance of the trace under cyclic permutations we have
\begin{align*}
\trace \left( Z_1^T B' \right) &= \trace \left( Q^+ P^{+^T} B' \right) = \trace \left( P^{+^T} B' Q^+ \right) \\
&= \trace \left( \underbrace{P^{+^T}P^+}_{= \mathbbm{1}} P^{+^T} B' Q^+ \underbrace{Q^{+^T}Q^+}_{=\mathbbm{1}}  \right) \\
&= \trace \left( Q^+P^{+^T} P^+P^{+^T} B' Q^+Q^{+^T} \right) \\
&\leq \Omega_* \left( Q^+ P^{+^T} \right) \Vert P^+P^{+^T} B' Q^+Q^{+^T}  \Vert_{\mbox{nuclear}} \\ 
&= \left\Vert P^+P^{+^T} B' Q^+Q^{+^T} \right\Vert_{\mbox{nuclear}},
\end{align*}
since $\Lambda_{\max} \left( P^+Q^{+^T} \right) =1$.

On the other hand, consider $\trace \left( Z_2^T B' \right)$. Again by the invariance of the trace under cyclic permutations we have
\begin{align*}
\trace \left( Z_2^T B' \right)&= \trace \left( \left( \mathbbm{1}_{q \times q } - Q^+Q^{+^T} \right) W^T \left(\mathbbm{1}_{p \times p} - P^+P^{+^T} \right) B' \right) \\
&= \trace \left(W^T \left( \mathbbm{1}_{p \times p } - P^+P^{+^T} \right) B' \left(\mathbbm{1}_{q \times q} - Q^+Q^{+^T} \right) \right) \\
&\leq \underbrace{\Lambda_{\max} \left( W \right)}_{=1} \left\Vert \left( \mathbbm{1}_{p \times p } - P^+P^{+^T} \right) B' \left(\mathbbm{1}_{q \times q} - Q^+Q^{+^T} \right) \right\Vert_{\mbox{nuclear}} \\
&= \underset{W: \Lambda_{\max} \left( W \right) = 1 }{\sup} \left\vert \trace \left( W^T \left( \mathbbm{1}_{p\times p} - P^+P^{+^T} \right) B' \left( \mathbbm{1}_{q \times q} - Q^+Q^{+^T} \right) \right) \right\vert.
\end{align*}
Hence, it is possible to find a $W$ such that $\Lambda_{\max} \left( W \right) =1 $ and such that
\begin{equation}
\trace \left( W^T B' \right) = \left\Vert \left( \mathbbm{1}_{p \times p } - P^+P^{+^T} \right) B' \left(\mathbbm{1}_{q \times q} - Q^+Q^{+^T} \right) \right\Vert_{\mbox{nuclear}}.
\end{equation}
Substituting $B'$ with $B'-B^+$ we have that
\begin{align*}
&\underset{Z \in \partial \left\Vert B^+ \right\Vert_{\mbox{nuclear}}}{\max} \trace \left( Z^T \left(B' - B^+ \right) \right)  \\
&=\underset{\Lambda_{\max} (W) =1}{\max} \left\lbrace \trace \left(  \left( \mathbbm{1}_{q \times q } - Q^+Q^{+^T} \right) W^T \left(\mathbbm{1}_{p \times p} - P^+P^{+^T} \right) \left(B'-B^+ \right) \right) \right. \\
&\phantom{=} \left. + \trace \left( Q^+P^{+^T} \left( B' - B^+ \right)  \right) \right\rbrace \\
\end{align*}
Hence,
\begin{align*}
&\left\Vert P^+P^{+^T} \left(B'-B^+\right) Q^+Q^{+^T} \right\Vert_{\mbox{nuclear}} - \left\Vert \left(\mathbbm{1} - P^+P^{+^T} \right) B' \left( \mathbbm{1} - Q^+Q^{+^T} \right) \right\Vert_{\mbox{nuclear}} \\
&\leq \sqrt{s} \left\Vert P^+P^{+^T} \left( B'-B^+ \right) Q^+Q^{+^T} \right\Vert_F - \Omega_{B^+}^- \left( B'\right) \\
&\leq \Omega_{B^+}^+ \left( B' - B^+ \right).
\end{align*}
For the second assertion we have
\begin{align*}
\left\Vert B' \right\Vert_{\mbox{nuclear}} &= \left\Vert P^+P^{+^T}B' + B'Q^+Q^{+^T} - P^+P^{+^T} B' Q^+Q^{+^T} \right.\\
&\phantom{=}\left.+ \left( \mathbbm{1}-P^+P^{+^T} \right) B'\left(\mathbbm{1}-Q^+Q^{+^T}\right) \right\Vert_{\mbox{nuclear}} \\
&\leq \left\Vert P^+P^{+^T} B' \right\Vert_{\mbox{nuclear}} + \left\Vert B' Q^+Q^{+^T} \right\Vert_{\mbox{nuclear}} \\
&\phantom{=} + \left\Vert P^+P^{+^T} B' Q^+Q^{+^T} \right\Vert_{\mbox{nuclear}} \\
&\phantom{=} + \left\Vert \left(\mathbbm{1} - P^+P^{+^T} \right)B' \left(\mathbbm{1} -Q^+Q^{+^T} \right) \right\Vert_{\mbox{nuclear}} \\
&\leq \sqrt{s} \left(\left\Vert P^+P^{+^T} B' \right\Vert_F + \left\Vert B' Q^+Q^{+^T} \right\Vert_F + \left\Vert P^+P^{+^T} B' Q^+Q^{+^T} \right\Vert_F \right) \\
&\phantom{=} + \left\Vert \left(\mathbbm{1} - P^+P^{+^T} \right)B' \left(\mathbbm{1} -Q^+Q^{+^T} \right) \right\Vert_{\mbox{nuclear}}.
\end{align*}
\end{proof}

Since the dual of $\underline{\Omega}$ may not be easy to deal with we bound it by the dual norm of the nuclear norm. By doing so, it will be possible to apply results about the tail of the maximum singular value of a sum of independent random matrices.
\begin{lemma}
\label{lemma:dualbound}
For the dual norm of $\underline{\Omega}$ we have that
\begin{equation*}
\underline{\Omega}_* \leq \Lambda_{\max}.
\end{equation*}
\end{lemma}

\begin{proof}
By Lemma 2.1 in the main text it is known that 
\begin{equation*}
\left\Vert \cdot \right\Vert_{\mbox{nuclear}} \leq \Omega_{B^+}^+ + \Omega_{B^+}^- = \underline{\Omega}.
\end{equation*}
Therefore, using Lemma 2.2 in the main text we have for the dual norms that
\begin{equation*}
\underline{\Omega}_* \leq \left\Vert \cdot \right\Vert_{\mbox{nuclear}^*} = \Lambda_{\max}.
\end{equation*}
\end{proof}

\subsection{Probability bounds for the empirical process}
\label{s:empiric}
To bound the expectation of $\Lambda_{\max}\left( \frac{1}{n} \sum_{i =1}^n \tilde{\varepsilon}_i X_i \right)$ where $\tilde{\varepsilon}_1, \dots, \tilde{\varepsilon}_n$ are i.i.d. Rademacher random variables independent of $(X_i, Y_i)_{i=1}^n$ we use Theorem 2.1 in the main text together with the fact that the $2$-Orlicz norm of a Rademacher random variable $\tilde{\varepsilon}$ is equal to
	\begin{equation*}
	\left\Vert \tilde{\varepsilon} \right\Vert_{\psi_2} = \sqrt{\frac{1}{\log 2}}.
	\end{equation*}
Now we check the assumptions of Theorem 2.1 in the main text.

For the case of matrix completion with uniform sampling we obtain that
\begin{equation*}
\mathbb{E} X_i X_i^T = \frac{1}{pq} \iota \iota^T,
\end{equation*}
where $\iota$ is a $q-$vector consisting of only $1$'s. It follows that
\begin{equation*}
\Lambda_{\max}\left(\mathbb{E} X_i X_i^T \right) = \Lambda_{\max} \left(\frac{1}{pq} \iota \iota^T \right)= \frac{1}{p}.
\end{equation*}
Analogously, we obtain that
\begin{equation*}
\Lambda_{\max} \left(\mathbb{E}X_i^T X_i \right) = \frac{1}{q}.
\end{equation*}
Notice that by the independence of the Rademacher sequence and the random matrices $X_i$ we have $\mathbb{E} \tilde{\varepsilon}_i X_i^T = \mathbb{E} \tilde{\varepsilon}_i \mathbb{E} X_i^T = 0$, $\mathbb{E} \tilde{\varepsilon}_i^2 X_i X_i^T = \mathbb{E} X_i X_i^T$, and $\mathbb{E} \tilde{\varepsilon}_i^2 X_i^T X_i = \mathbb{E} X_i^T X_i$. Moreover, since $\Lambda_{\max}(\cdot)$ is a norm we have
\begin{equation*}
\Lambda_{\max} (\tilde{\varepsilon}_i X_i) = \left\vert \tilde{\varepsilon}_i \right\vert \Lambda_{\max} (X_i) \leq \left\vert \tilde{\varepsilon}_i \right\vert, \forall i.
\end{equation*}
Hence,
\begin{equation*}
\left\Vert \Lambda_{\max}^2(\tilde{\varepsilon}_i X_i) \right\Vert_{\psi_2} \leq \left\Vert \tilde{\varepsilon} \right\Vert_{\psi_2} = \sqrt{\frac{1}{\log 2}}, \forall i.
\end{equation*}

\begin{lemma}
	\label{lemma:lipschitzdiff}
	\begin{enumerate}
		\item Suppose that $\rho$ is differentiable and Lipschitz continuous with constant $L$.  Suppose further that $\dot{\rho}(x)x$ is Lipschitz continuous with Lipschitz constant $\tilde{L}$. Assume that for a constant $K_{\underline{\Omega}}$
\begin{equation*}
\underset{1 \leq i \leq n}{\max} \underline{\Omega}_* (X_i) \leq K_{\underline{\Omega}}.
\end{equation*}
Define for all $M > 0$
\begin{equation}
\mathbf{Z}_M := \underset{B' \in \mathcal{B} : \underline{\Omega}(B-B')\leq M}{\sup}  \left\vert \trace\left( \left(\dot{R}_n (B') - \dot{R}(B') \right)^T \left(B - B' \right) \right) \right\vert.
\end{equation}
Then we have for a constant $C_0 >0$ and $\tilde{\lambda}_{1} = 8 \eta \tilde{L} p \log(p+q)/(3n)$ that
\begin{align*}
 &\mathbf{Z}_M \\
 &\leq \tilde{L} M \left( (8 C_0 + \sqrt{2})\sqrt{\frac{\log (p+q)}{nq}} +8C_0 \sqrt{\log(1+q)} \frac{\log(p+q)}{n} \right) + \tilde{\lambda}_{1}
\end{align*}
with probability at least $1-\exp(-p \log(p+q))$.

\item Let $\rho$ be a Lipschitz continuous function with Lipschitz constant $L$. Assume that for a constant $K_{\underline{\Omega}}$
\begin{equation*}
\underset{1 \leq i \leq n}{\max} \ \underline{\Omega}_* \left( X_i \right) \leq K_{\underline{\Omega}}.
\end{equation*}
Define for all $M>0$
\begin{equation*}
\mathbf{Z}_M := \underset{B' \in \mathcal{B}: \underline{\Omega}(B' - B) \leq M}{\sup} \left\vert \left[R_n(B') - R(B') \right] - \left[R_n(B) - R(B) \right] \right\vert.
\end{equation*}
Then we have for a constant $C_0 > 0$ and $\tilde{\lambda}_{2} = 8 \eta L p \log(p+q)/(3n)$ that
\begin{align*}
	&\mathbf{Z}_M \\
	&\leq LM \left( (8 C_0 + \sqrt{2})\sqrt{\frac{\log (p+q)}{nq}} +8C_0 \sqrt{\log(1+q)} \frac{\log(p+q)}{n} \right) + \tilde{\lambda}_{2}
\end{align*}
with probability at least $1-\exp(-p \log(p+q))$.
	\end{enumerate}
\end{lemma}

\begin{proof}
The proof of this lemma is based on the Symmetrization Theorem, on the Contraction Theorem, on the dual norm inequality, and on Bousquet's Concentration Theorem. We use the notation $\rho(X_i B, Y_i) := \rho(Y_i - \trace(X_i B))$ and $\dot{\rho} (X_i B, Y_i) := \dot{\rho} (Y_i - \trace(X_i B))$.

We have
\begin{align}
&\mathbb{E} \left[ \underset{B' \in \mathcal{B} : \underline{\Omega}(B-B')\leq M}{\sup}  \left\vert \trace \left( \left(\dot{R}_n (B')-\dot{R} (B') \right)^T \left(B - B' \right)  \right) \right\vert \right] \nonumber \\
&= \mathbb{E} \Bigg[\underset{B' \in \mathcal{B}: \underline{\Omega}(B - B') \leq M}{\sup}  \left\vert\frac{1}{n} \sum_{i=1}^n \trace \left( \left(B - B' \right)^T \left( X_i^T \dot{\rho} (X_i B', Y_i) \right. \right. \right. \\
&\left.\left. \left. \phantom{\leq}- \mathbb{E} X_i^T \dot{\rho} ( X_i B', Y_i) \right) \right) \right\vert \Bigg] \nonumber \\
&\leq 2 \mathbb{E} \left[\underset{B' \in \mathcal{B}: \underline{\Omega}(B- B') \leq M}{\sup}  \left\vert \frac{1}{n} \sum_{i=1}^n \tilde{\varepsilon}_i \trace \left( (B - B')^T X_i^T \dot{\rho}(X_i B', Y_i) \right) \right\vert \right] \nonumber \\
&\leq 4 \tilde{L}\underline{\Omega}(B - B') \mathbb{E} \underline{\Omega}_* \left(\sum_{i=1}^n \tilde{\varepsilon}_i X_i \right)/n, \nonumber \\
&\leq 4\tilde{L}M \mathbb{E} \Omega_* \left(\sum_{i=1}^n \tilde{\varepsilon}_i X_i \right)/n, \ \mbox{since} \ \Omega \leq \underline{\Omega} \Rightarrow \underline{\Omega}_* \leq \Omega_*.
\label{eqn:symmandcontr}
\end{align}
The first inequality follows from Theorem \ref{thm:symmetrization}, the second from Theorem \ref{thm:contraction} and the dual norm inequality.
 With $S=1/\sqrt{q}$ and $K =\sqrt{1/\log 2}$ in Theorem 2.1 in the main text together with the concavity of the logarithm 
\begin{align*}
\sqrt{\log\left( \sqrt{\frac{q}{\log 2}} \right)} = \sqrt{\frac{1}{2} \log q + \frac{1}{2} \log\left( \frac{1}{\log 2}\right)} \leq \sqrt{\log\left( \frac{q}{2} + \frac{1}{2\log 2} \right) } \\
 \leq \sqrt{\log \left( q + 1 \right)}
\end{align*}
we obtain integrating the tail that the expectation of the largest singular value of the sum of masks is bounded by
\begin{equation}
\mathbb{E} \Lambda_{\max} \left( \sum_{i=1}^n \tilde{\varepsilon}_i X_i \right)/n \leq C_0 \left( \sqrt{\frac{\log(p+q)}{nq}} + \sqrt{\log\left( q + 1 \right) } \frac{\log(p+q)}{n} \right).
\end{equation}

Hence, defining
\begin{equation}
f(X_i B' ) := \trace \left(\left(B - B' \right)^T X_i^T \dot{\rho} (X_i B', Y_i ) \right)
\end{equation}
we have using the Lipschitz continuity of the loss function and the fact that the $X_i \in \chi$ are i.i.d.
\begin{align*}
&\underset{B' \in \mathcal{B} : \underline{\Omega}(B-B') \leq M}{\sup} \Var(f(X_1B'))  \\
&= \underset{B' \in \mathcal{B}: \underline{\Omega}(B-B') \leq M}{\sup} \Var \left( \sum_{l=1}^q \sum_{j = 1}^p X_{1_{lj}}(B_{jl} - B'_{jl})\dot{\rho}(X_1 B', Y_i)) \right)\\
&\leq \underset{B' \in \mathcal{B} : \underline{\Omega}(B-B') \leq M}{\sup} L^2 \mathbb{E} \left[ \left(\sum_{l=1}^q \sum_{j=1}^p X_{1_{lj}} (B_{jl} - B_{jl}') \right)^2 \right] \\
&= \underset{B' \in \mathcal{B} : \underline{\Omega} (B- B') \leq M }{\sup} L^2 \mathbb{E} \left[ \sum_{l=1}^q \sum_{j=1}^p X_{1_{lj}}^2 (B_{jl} - B_{jl}')^2  \right] \\
&= \underset{B' \in \mathcal{B}: \underline{\Omega}(B-B')\leq M}{\sup} \frac{\tilde{L}^2}{pq} \Vert B_{jl} - B_{jl}' \Vert_F^2 \\
&\leq \frac{\tilde{L}^2 M^2}{pq} =: T^2.
\end{align*}
Therefore, using
\begin{equation}
\Vert f \Vert_\infty \leq 2 \eta \tilde{L} =: D
\end{equation}
we obtain from Bousquet's Concentration Theorem \ref{thm:bousquet} that for all $t>0$
\begin{align*}
\mathbb{P} \left(\mathbf{Z}_M \geq 8 \tilde{L}M C_0 \left(\sqrt{\frac{\log(p+q)}{nq}} + \sqrt{\log \left(q+1\right)} \frac{\log(p+q)}{n} \right) \right. \\
\left. + \frac{M \tilde{L}}{\sqrt{pq}}\sqrt{\frac{2t}{n}} + \frac{8t \eta \tilde{L} }{3n}\right) \\
 \leq \exp(-t).
\end{align*}
Replacing $t$ by $p \log(p+q)$ we obtain
\begin{align*}
\mathbb{P} \left( \mathbf{Z}_M \geq M \tilde{L} \left( (8C_0 + \sqrt{2}) \sqrt{\frac{\log(p+q)}{nq}} + 8C_0 \sqrt{\log (q+1)} \frac{\log(p+q)}{n} \right) \right. \\
\left. + \frac{8 \eta \tilde{L} p \log(p+q)}{3n} \right) \leq \exp(-p \log(p+q)).
\end{align*}

For the second claim we proceed similarly. We have
\begin{align*}
&\mathbb{E} \left[ \underset{B' \in \mathcal{B} : \underline{\Omega}(B' - B) \leq M}{\sup} \left\vert \left[ R_n \left(B' \right) - R \left(B' \right) \right] - \left[R_n \left(B \right) - R\left( B \right) \right] \right\vert \right] \\
&=\mathbb{E} \left[ \underset{B' \in \mathcal{B} : \underline{\Omega}(B' - B) \leq M}{\sup} \left\vert \left[\frac{1}{n} \sum_{i = 1}^n \left\lbrace \rho \left(X_i B', Y_i \right)  - \mathbb{E} \rho \left( X_i B', Y_i \right) \right\rbrace  \right] \right. \right. \\
&\left.\left. \phantom{\mathbb{E} \dots \dots \dots}- \left[\frac{1}{n} \sum_{i = 1}^n \left\lbrace \rho \left(X_iB, Y_i \right) - \mathbb{E}\rho \left( X_i B, Y_i \right) \right\rbrace \right] \right\vert \right] \\
&= \mathbb{E} \left[ \underset{B' \in \mathcal{B} : \underline{\Omega}(B' - B) \leq M}{\sup} \left\vert \left[\frac{1}{n} \sum_{i = 1}^n \left\lbrace \rho \left(X_i B', Y_i \right)  - \rho \left(X_i B, Y_i \right) \right\rbrace  \right] \right. \right. \\ 
&\left.\left. \phantom{\mathbb{E} \dots \dots \dots}- \left[\frac{1}{n} \sum_{i = 1}^n \left\lbrace  \mathbb{E} \rho \left( X_i B', Y_i\right)  - \mathbb{E}\rho \left( X_i B, Y_i \right) \right\rbrace \right] \right\vert \right] \\
&\leq 2 \mathbb{E} \left[ \underset{B' \in \mathcal{B} : \underline{\Omega}(B' - B) \leq M}{\sup} \left\vert \frac{1}{n} \sum_{i=1}^n \tilde{\varepsilon}_i \left(\rho \left(X_i B', Y_i \right) - \rho \left( X_i B, Y_i \right) \right) \right\vert \right]\\
&\leq 4 \mathbb{E} \left[\underset{B' \in \mathcal{B} : \underline{\Omega}(B' - B) \leq M}{\sup} \left\vert \frac{1}{n} \sum_{i=1}^n L \tilde{\varepsilon}_i \trace \left(X_i \left(B'-B \right) \right) \right\vert  \right] \\
&= 4L \mathbb{E} \left[ \underset{B' \in \mathcal{B} : \underline{\Omega}(B' - B) \leq M}{\sup} \left\vert \trace \left( \left\lbrace \frac{1}{n} \sum_{i =1}^n \tilde{\varepsilon}_i X_i \right\rbrace \left\lbrace B' -B \right\rbrace \right) \right\vert \right] \\
&\leq 4 L\mathbb{E}  \underline{\Omega}_* \left( \frac{1}{n} \sum_{i=1}^n \tilde{\varepsilon}_i X_i \right) \underline{\Omega} \left( B'-B \right), \ \mbox{by the dual norm inequality}  \\
&\leq 4L M \mathbb{E}  \Lambda_{\max} \left( \frac{1}{n} \sum_{i=1}^n \tilde{\varepsilon}_i X_i \right) \\
&\leq 4 C_0L M \left( \sqrt{\frac{\log(p+q)}{nq}} + \sqrt{\log\left( q+1\right) } \frac{\log(p+q)}{n} \right), \ \mbox{as before.}
\end{align*}
The first inequality follows from Theorem \ref{thm:symmetrization}, the second from Theorem \ref{thm:contraction}. Moreover, we have for all $i = 1, \dots, n$ that
\begin{align*}
\left\vert \rho(Y_i - \trace(X_i B')) - \rho(Y_i - \trace(X_i B)) \right\vert \leq L \left\vert \trace \left(X_i \left(B'-B \right) \right) \right\vert \\
\leq 2 \eta L =: D.
\end{align*}
In view of applying Bousquet's Concentration Theorem \ref{thm:bousquet} we have with a similar calculation as for the first claim
\begin{align*}
&\mathbb{E}\left[\left( \rho( Y_i - \trace ( X_i B') ) - \rho( Y_i - \trace(X_i B) ) \right)^2 \right] \\
&\leq L^2 \mathbb{E} \left[\trace^2 \left(X_i(B-B') \right) \right] \\
&\leq \frac{M^2 L^2}{pq} =: T^2
\end{align*}
Replacing $t$ by $p \log(p+q)$ we obtain 
\begin{align*}
\mathbb{P} \left(\mathbf{Z}_M \geq  8 C_0 L M \left( \sqrt{\frac{\log(p+q)}{nq}} + \sqrt{\log\left(q+1\right) } \frac{\log(p+q)}{n}\right) \right. \\
\left. +\frac{ML}{\sqrt{pq}} \sqrt{\frac{2 p \log(p+q)}{n}} + \tilde{\lambda}_{2} \right) \\
 \leq \exp \left(- p \log(p+q) \right).
\end{align*}
\end{proof}

To obtain a uniform bound for all $B \in \mathcal{B}$ we use the peeling device given in \cite{van2000applications}.
\begin{lemma}
\label{lemma:peeling}
\begin{enumerate}
\item Let $\tilde{L}$ and $\tilde{\lambda}_1$ be as in Lemma \ref{lemma:lipschitzdiff}. Define
\begin{equation*}
\lambda_{\varepsilon,1} = \tilde{L} \left((8C_0 + \sqrt{2}) \sqrt{\frac{\log(p+q)}{nq}} + 8C_0 \sqrt{\log(1+q)} \frac{\log(p+q)}{n}\right)
\end{equation*}
	 We have for any fixed $B \in \mathcal{B}$
\begin{align*}
&\mathbb{P} \Biggl(\exists B' \in \mathcal{B} : \left\vert \trace \left( \left(\dot{R}_n(B') - \dot{R}(B') \right)^T \left( B' - B \right) \right) \right\vert  \\
&\phantom{\dots \dots \dots }> 2 \lambda_{\varepsilon,1}( \underline{\Omega}(B'-B) + 1) + \tilde{\lambda}_{1} \Biggr)\\
&\leq (j_0+2) \exp(-p \log(p+q)).
\end{align*}
	\item Let $L$ and $\tilde{\lambda}_2$ be as in Lemma \ref{lemma:lipschitzdiff}. Define
	\begin{equation*}
\lambda_{\varepsilon,2} = L \left((8C_0 + \sqrt{2}) \sqrt{\frac{\log(p+q)}{nq}} + 8C_0 \sqrt{\log(1+q)} \frac{\log(p+q)}{n}\right)
\end{equation*}
	 We have for any fixed $B \in \mathcal{B}$
	\begin{align*}
	&\mathbb{P} (\exists B' \in \mathcal{B} : \left\vert (R_n(B') - R(B')) - (R_n(B) - R(B)) \right\vert   \\
	& \phantom{\dots \dots \dots \dots}> 2 \lambda_{\varepsilon,2}(\underline{\Omega}(B'-B) + 1)+ \tilde{\lambda}_{\textcolor{red}{2}} ) \\
	&\leq (j_0 + 2) \exp(-p\log(p+q)).
	\end{align*}
\end{enumerate}
\end{lemma}

\begin{proof}
We subdivide the set $\mathcal{B}$ as follows for a fixed $B \in \mathcal{B}$
\begin{align*}
\mathcal{B} &= \left\lbrace B' \in \mathcal{B} : \underline{\Omega} ( B-B') \leq 1 \right\rbrace \\
&\phantom{\dots \dots}\cup \left\lbrace B' \in \mathcal{B} : 1 < \underline{\Omega}(B- B') \leq 14 q \sqrt{pq} \eta \right\rbrace.
\end{align*}

Suppose we are on the first set, then
\begin{align*}
&\left\lbrace \exists B' : \underline{\Omega}(B-B') \leq 1, \left\vert \trace \left( (\dot{R}_n(B') - \dot{R}(B'))^T(B-B') \right) \right\vert \right.\\
 &\phantom{\dots \dots \dots \dots \dots \dots \dots \dots}\left. > 2 \lambda_{\varepsilon,1} (\underline{\Omega}(B-B') +1) + \tilde{\lambda}_{1} \right\rbrace \\
 &\subset \left\lbrace \exists B' : \underline{\Omega} (B-B') \leq 1, \left\vert \trace \left( (\dot{R}_n(B') - \dot{R}(B'))^T(B-B') \right) \right\vert \right. \\
 &\phantom{\lbrace \exists B' : \leq 1, \left\vert \trace \left( (\dot{R}_n(B') - \dot{R}(B'))^T(B-B') \right) \right\vert}\left. > \lambda_{\varepsilon,1} + \tilde{\lambda}_{1} \right\rbrace.
\end{align*}
Therefore, by Lemma \ref{lemma:lipschitzdiff} we conclude that
\begin{align*}
\mathbb{P} \left( \exists B' : \underline{\Omega}(B-B') \leq 1, \left\vert \trace \left( (\dot{R}_n(B')- \dot{R}(B'))^T(B-B') \right) \right\vert \right.  \\
\left. > 2 \lambda_{\varepsilon,1}(\underline{\Omega}(B-B') + 1) + \tilde{\lambda}_{1} \right) \\
\leq \exp(-p \log(p+q)).
\end{align*}
We then consider the set $\left\lbrace B' \in \mathcal{B} : 1< \underline{\Omega}(B-B') \leq 14 q \sqrt{pq} \eta \right\rbrace$. We first refine it by choosing $j_0$ as the smallest integer such that $j_0 + 1 > \log_2 (14 q \sqrt{pq} \eta)$. This leads us to
\begin{align*}
&\left\lbrace B' \in \mathcal{B} : 1< \underline{\Omega}(B-B') \leq 14 q \sqrt{pq} \eta \right\rbrace \\
&\subset \bigcup_{j = 0}^{j_0} \underbrace{\left\lbrace B' : 2^j <\underline{\Omega}(B-B') \leq 2^{j+1} \right\rbrace}_{=: \mathcal{B}_j}.
\end{align*}
For one $j$ we have for the event
\begin{align*}
\left\lbrace \exists B' \in \mathcal{B}_j : \left\vert \trace \left( (\dot{R}_n(B') - \dot{R}(B'))^T(B'-B) \right) \right\vert \right. \\
\left. > 2 \lambda_{\varepsilon,1} (\underline{\Omega}(B-B') +1 ) + \tilde{\lambda}_{\textcolor{red}{1}} \right\rbrace \\
\subset \left\lbrace \exists B' \in \mathcal{B}_j : \left\vert \trace \left( (\dot{R}_n(B') - \dot{R}(B'))^T(B'-B) \right) \right\vert > 2^{j+1} \lambda_{\varepsilon,1} + \tilde{\lambda}_{\textcolor{red}{1}} \right\rbrace.
\end{align*}
By the first claim in Lemma \ref{lemma:lipschitzdiff} we can conclude that
\begin{align*}
\mathbb{P} \left(\exists B' \in \mathcal{B}_j :  \left\vert \trace \left( (\dot{R}_n(B') - \dot{R}(B'))^T(B'-B) \right) \right\vert \right. \\
\left. > 2 \lambda_{\varepsilon,1} (\underline{\Omega}(B-B') + 1) + \tilde{\lambda}_{\textcolor{red}{1}} \right)\\
\leq \exp(-p \log(p+q)).
\end{align*}
To keep the notation clean we define the event
\begin{align*}
&\mathcal{C} = \Big\lbrace \exists B' \in \mathcal{B}: 1 < \underline{\Omega}(B-B') \leq 14 q \sqrt{pq} \eta,  \\ 
&\phantom{\dots}\left\vert \trace \left((\dot{R}_n(B') - \dot{R}(B'))^T(B-B') \right) \right\vert  > 2 \lambda_{\varepsilon,1}(\underline{\Omega}(B-B') +1) + \tilde{\lambda}_{1}  \Big\rbrace
\end{align*}
and  for all $j = 0, \dots, j_0$ the events
\begin{align*}
&\mathcal{C}_j = \left\lbrace \exists B' \in \mathcal{B}_j : \left\vert \trace \left( (\dot{R}_n(B') - \dot{R}(B'))^T(B' - B) \right) \right\vert \right. \\
&\phantom{\dots \dots \dots \dots \dots \dots \dots \dots}\left. > 2^{j+1} \lambda_{\varepsilon,1} + \tilde{\lambda}_{1} \right\rbrace.
\end{align*}
Then $\mathcal{C} \subset \bigcup_{j=0}^{j_0} \mathcal{C}_j$.
Therefore, by the union bound
\begin{align*}
\mathbb{P} \left( \mathcal{C} \right) \leq \mathbb{P} \left( \bigcup_{j=0}^{j_0} \mathcal{C}_j \right) \leq \sum_{j=0}^{j_0} \mathbb{P} \left( \mathcal{C}_j \right) \leq (j_0+1) \exp(-p \log(p+q)).
\end{align*}
The second claim follows by an analogous reasoning.
\end{proof}


The proof of Lemma \ref{lemma:lipschitzdiff} requires the Symmetrization Theorem and the Contraction Theorem (or Contraction Principle). The first theorem allows us to reduce the case of non-centred random variables to the case of mean zero random variables. A main tool for this type of reduction is a sequence of Rademacher random variables. The second theorem, the Contraction Principle, is used to compare the limit behaviour of two series of Rademacher random variables with different coefficients. We state these results for the sake of completeness.

\begin{theorem}[Symmetrization of Expectations, \cite{van1996weak}]
\label{thm:symmetrization}
Consider $X_1, \dots, X_n$ independent matrices in $\chi$ and let $\mathcal{F}$ be a class of real-valued functions on $\chi$. Let $\tilde{\varepsilon}_1, \dots, \tilde{\varepsilon}_n$ be a Rademacher sequence independent of $X_1, \dots, X_n$, then
\begin{equation}
\mathbb{E} \left[\underset{f \in \mathcal{F}}{\sup} \ \left\vert \sum_{i=1}^n \left( f(X_i) - \mathbb{E} f (X_i) \right) \right\vert \right] \leq 2 \mathbb{E} \left[\underset{f \in \mathcal{F}}{\sup} \ \left\vert \sum_{i=1}^n \tilde{\varepsilon}_i f(X_i) \right\vert \right].
\end{equation}
\end{theorem}

\begin{theorem}[Contraction Theorem, \cite{ledoux1991probability}]
\label{thm:contraction}
Consider the non-random elements $x_1, \dots, x_n$ of $\chi$. Let $\mathcal{F}$ be a class of real-valued functions on $\chi$. Consider the Lipschitz continuous functions $\rho_i : \mathbb{R} \rightarrow \mathbb{R}$ with Lipschitz constant $L$, i.e.
\begin{equation*}
\left\vert \rho_i (\mu) - \rho_i (\tilde{\mu}) \right\vert \leq L \left\vert \mu- \tilde{\mu} \right\vert, \ \mbox{for all} \ \mu, \tilde{\mu} \in \mathbb{R}.
\end{equation*}
Let $\tilde{\varepsilon}_1, \dots, \tilde{\varepsilon}_n$ be a Rademacher sequence . Then for any function $f^* : \chi \rightarrow \mathbb{R}$, we have
\begin{equation}
\mathbb{E} \left[\underset{f \in \mathcal{F}}{\sup}  \left\vert \sum_{i =1}^n \tilde{\varepsilon}_i \left\lbrace \rho_i (f(x_i)) - \rho_i (f^*(x_i)) \right\rbrace \right\vert \right] \leq 2 \mathbb{E} \left[ L \underset{f \in \mathcal{F}}{\sup}  \left\vert \sum_{i =1}^n \tilde{\varepsilon}_i \left(f(x_i) - f^*(x_i) \right) \right\vert \right].
\end{equation}
\end{theorem}

\begin{theorem}[Bousquet's Concentration Theorem, Bousquet (2002)]
	\label{thm:bousquet}
	Suppose that for all $i = 1, \dots, n$ and for all $f \in \mathcal{F}$
	\begin{equation*}
	\mathbb{E} \left[ f(X_i) \right] = 0, \frac{1}{n} \sum_{i=1}^n \underset{f \in \mathcal F}{\sup} \ \mathbb{E} \left[ f^2(X_i) \right] \leq T^2.
	\end{equation*}
	Assume further for a constant $D > 0$ and for all $f \in \mathcal{F}$ that
	\begin{equation*}
	\Vert f \Vert_\infty \leq D.
	\end{equation*}
	Define
	\begin{equation*}
	\mathbf{Z} := \underset{f \in \mathcal{F}}{\sup} \ \left\vert \frac{1}{n} \sum_{i=1}^n \left( f(X_i) - \mathbb{E} \left[ f(X_i) \right] \right) \right\vert.
	\end{equation*}
	Then we have for all $t>0$
	\begin{equation*}
	\mathbb{P} \left( \mathbf{Z} \geq 2 \mathbb{E} \mathbf{Z} + T \sqrt{\frac{2t}{n}} + \frac{4 t D}{3n} \right) \leq \exp (-t).
	\end{equation*}
\end{theorem}


\subsection{On the distribution of the errors}
\label{ss:distoferrors}

In this section we discuss the consequences arising from the assumption that the distribution of the errors is symmetric around $0$. Let $F$ be the distribution function of the errors. For purposes of illustration, we discuss the location model. The case of low-rank matrix estimation then follows easily. The location model is as follows
\begin{equation}
X = \mu^* + \varepsilon,
\end{equation}
where $\mu^*$ is some fixed real number such that $\vert \mu^* \vert \leq \eta$ and $\varepsilon$ is additive noise with symmetric (around $0$) distribution.
\begin{lemma}
Assume that $f$ is the density with respect to Lebesgue measure of the errors. Suppose further that $f(u) > 0$ for all $\vert u \vert \leq 2 \eta$. Define
\begin{equation}
\mu^0_1 = \underset{\vert \mu \vert \leq \eta}{\arg \min} \ \mathbb{E} \left[ \vert X - \mu \vert \right].
\end{equation}
Then,
\begin{equation}
\mu^* = \mu^0_1.
\end{equation}
\end{lemma}
\begin{proof}
We first notice that $\mu^0_1$ is the median of the distribution of $X$. Since the distribution is continuous and the density positive everywhere the median is unique. \\
Since the distribution of the errors is symmetric around $0$ the distribution of $X$ is symmetric around $\mu^*$. This implies that $\mu^*$ must be the median. But since the median is unique we must have $\mu^*= \mu^0_1$.
\end{proof}

\begin{lemma}
Assume that the distribution function of the errors satisfies
\begin{equation}
F(u+\kappa) - F(u - \kappa) \geq 1/C_1^2 \ \text{for all} \ \vert u \vert \leq 2\eta \ \text{and} \ \kappa \leq 2 \eta.
\end{equation}
Define
\begin{equation}
\mu^0_2 = \underset{\vert \mu \vert \leq \eta}{\arg \min} \ \mathbb{E} \left[ \rho_H (X - \mu) \right],
\end{equation}
where $\rho_H$ is the Huber loss as defined in the main text. Then,
\begin{equation}
\mu^* = \mu^0_2.
\end{equation}
\end{lemma}
\begin{proof}
We notice that the first derivative of the Huber loss evaluated in $X - \mu^0_2$ is given by
\begin{equation}
\dot{\rho}_H(X - \mu^0_2) = \left\lbrace \begin{array}{ll}
-2(X - \mu^0_2), & \text{if} \ \vert X - \mu^0_2 \vert \leq \kappa \\
-2 \kappa, & \text{if} \ X - \mu^0_2 > \kappa \\
2 \kappa, & \text{if} \ X - \mu^0_2 < - \kappa.
\end{array} \right.
\end{equation}
By straightforward computation 
\begin{align*}
\mathbb{E} \left[ \dot{\rho}_H (X - \mu_2^0) \right] &= 2 \int_{- \kappa + \mu^0_2 - \mu^*}^{\kappa + \mu^0_2 - \mu^*} F(\varepsilon) d \varepsilon - 2 \kappa = 0.
\end{align*}
We now use the symmetry of the distribution of the errors. This translates to
\begin{equation}
F(\varepsilon) = 1 - F(-\varepsilon).
\end{equation}
This implies
\begin{align*}
\int_{- \kappa + \mu^* - \mu_2^0}^{\kappa + \mu^* - \mu_2^0} F(\tilde{\varepsilon}) d \tilde{\varepsilon} = \kappa.
\end{align*}
On the other hand, we also have that
\begin{align*}
\int_{- \kappa + \mu_2^0 - \mu^*}^{\kappa + \mu^0_2 - \mu^*} F(\varepsilon) d\varepsilon = \kappa.
\end{align*}
By changing the variables in the previous integrals we arrive at
\begin{align*}
\int_{- \kappa}^{\kappa} F(\varepsilon + \mu_2^0 - \mu^*) - F(\varepsilon + \mu^* - \mu^0_2) d \varepsilon = 0.
\end{align*}
The previous integral is always larger than $0$ by assumption unless $\mu^* = \mu^0_2$.
\end{proof}

\subsection{Proof of Lemma 4.1 in the main text}

\begin{proof}[Proof of Lemma 4.1 in the main text]
	The proof of this lemma is analogous to the first part of the proof of Corollary 2 in \cite{negahban2012restricted}.
	Take $\sigma > 0$ such that $$s= \max \left\lbrace k \in \left\lbrace 1,\dots,q \right\rbrace \vert \Lambda_k^0 > \sigma \right\rbrace.$$ Then, we have
	\begin{align*}
	\left\Vert B^- \right\Vert_{\mbox{nuclear}} &= \sum_{k=s+1}^q \Lambda_k\\
	&= \sigma \sum_{k=s+1}^q \frac{\Lambda_k}{\sigma} \\
	&\leq \sigma \sum_{k=s+1}^q \left( \frac{\Lambda_k}{\sigma} \right)^r \\
	&\leq \sigma^{1-r} \rho_r^r.
	\end{align*}
	Moreover, we have
	\begin{equation*}
	s \sigma^r = \sigma^r \sum_{k =1}^s \mathbbm{1}_{\left\lbrace \Lambda_k^0> \sigma \right\rbrace} = \sigma^r \sum_{k=1}^s \mathbbm{1}_{\left\lbrace \left( \frac{\Lambda_k^0}{\sigma} \right)^r > 1 \right\rbrace} \leq \sigma^r \sum_{k=1}^s \left(\frac{\Lambda_k^0}{\sigma} \right)^r \leq \rho_r^r,
	\end{equation*}
	where $\Lambda_k^0$ denotes a singular value of the matrix $B^0$.
	Therefore we obtain
	\begin{equation*}
	s \leq \sigma^{-r} \rho_r^r.
	\end{equation*}
\end{proof}

\bibliographystyle{plainnat}
\bibliography{myreferences}

%
%
%
%
%
%

\end{document}